\documentclass[10pt]{article}

\usepackage{amsmath,amssymb,amsfonts,amsthm,amscd,eucal,amssymb}

\usepackage[dvips]{color}           % allows use of colors in LaTeX
\usepackage{varioref}
\usepackage{amsthm}
\usepackage[latin1]{inputenc}
\usepackage[english]{babel}
\usepackage{color}                          % Farbiger Text
\usepackage{mathbbol}
\theoremstyle{thmbreak}
\newtheorem{thm}{Theorem}

\numberwithin{equation}{section}

\catcode`\@=11
\long\def\@makefntext#1{
\protect\noindent \hbox to 3.2pt {\hskip-.9pt
$^{{\eightrm\@thefnmark}}$\hfil}#1\hfill}       %CAN BE USED

\def\@makefnmark{\hbox to 0pt{$^{\@thefnmark}$\hss}}    %ORIGINAL

\font\eightrm=cmr8

\newcommand{\smalllineskip}{\baselineskip=10pt}

\def\pmb#1{\setbox0=\hbox{#1}
    \kern-.025em\copy0\kern-\wd0
    \kern.05em\copy0\kern-\wd0
    \kern-.025em\raise.0433em\box0}

\def\ps@myheadings{\let\@mkboth\@gobbletwo
  \def\@oddhead{{\slshape\rightmark}\hfil{\footnotesize\thepage}}%{\hfill\hbox{}\rightmark}
  \def\@oddfoot{}
  \def\@evenhead{{\footnotesize\thepage}\hfil\slshape\leftmark}%{\leftmark\hbox{}\hfill}
  \def\@evenfoot{}
  \def\sectionmark##1{}\def\subsectionmark##1{}
}

\oddsidemargin=\evensidemargin
\newcommand{\copyrightheading}[1]
    {\vspace*{-2.5cm}\smalllineskip{\flushleft
    {\footnotesize Dipartimento di Matematica, Universit\`a di Trento}\\
    {\footnotesize Preprint UTM - #1}\\
    {\tiny \textbf{\jobname.tex}~--~\number\day/\number\month/\number\year~--~\timenow}%
    }}
\newcommand{\timenow}{%
  \@tempcnta=\time \divide\@tempcnta by 60 \number\@tempcnta:\multiply
  \@tempcnta by 60 \@tempcntb=\time \advance\@tempcntb by -\@tempcnta
  \ifnum\@tempcntb <10 0\number\@tempcntb\else\number\@tempcntb\fi}

\newcommand{\fcaption}[1]{
        \refstepcounter{figure}
        \setbox\@tempboxa = \hbox{\footnotesize Fig.~\thefigure. #1}
        \ifdim \wd\@tempboxa > 5in
           {\begin{center}
        \parbox{5in}{\footnotesize\smalllineskip Fig.~\thefigure. #1}
            \end{center}}
        \else
             {\begin{center}
             {\footnotesize Fig.~\thefigure. #1}
              \end{center}}
        \fi}

\newcommand{\tcaption}[1]{
        \refstepcounter{table}
        \setbox\@tempboxa = \hbox{\footnotesize Table~\thetable. #1}
        \ifdim \wd\@tempboxa > 5in
           {\begin{center}
        \parbox{5in}{\footnotesize\smalllineskip Table~\thetable. #1}
            \end{center}}
        \else
             {\begin{center}
             {\footnotesize Table~\thetable. #1}
              \end{center}}
        \fi}

\newtheorem{theorem}{Theorem}[section] %WHICH HAS SECTION NUMBER

\newtheorem{proposition}[theorem]{Proposition}
\newtheorem{definition}[theorem]{Definition}

\newtheorem{lemma}[theorem]{Lemma}

\newtheorem{hypothesis}[theorem]{Hypothesis}
\newtheorem{remark}[theorem]{Remark}
\newtheorem{example}[theorem]{Example}

\numberwithin{equation}{section}

\newcommand\N{\mathbb N}

\renewcommand\qed{\par\hfill\hbox{${\vcenter{\vbox{         %HOLLOW SQUARE
   \hrule height 0.4pt\hbox{\vrule width 0.4pt height 6pt
   \kern5pt\vrule width 0.4pt}\hrule height 0.4pt}}}$}}

\newcommand\R{\mathbb R}

\newcommand{\ep}{\varepsilon}
\newcommand{\EE}{\mathbb{E}}
\newcommand{\bN}{\mathbb{N}}
\newcommand{\PP}{\mathbb{P}}
\newcommand{\dd}{{\rm d}}

\begin{document}

%\begin{frontmatter}

  \title{Small noise asymptotic expansions for stochastic PDE's driven by dissipative nonlinearity and L\'evy noise}
  \def\lhead{S.\ ALBEVERIO, B. Smii, E.\ MASTROGIACOMO} %% for author's initials and surname
  \def\rhead{Small noise asymptotics for dissipative stochastic PDE's with Levy noise} %% for abbreviated title.

%\it  \\ Dept. Appl. Mathematics, University of Bonn, HCM; SFB611, BiBoS, IZKS \\
%%
% \it \\ NeSt Project, University of Trento, Dept.Math. \\

\author{Sergio Albeverio${}^\sharp{}$\thanks{Dept. Appl. Mathematics, University of Bonn, HCM; SFB611, BiBoS, IZKS. {\tt albeverio@uni-bonn.de}} %
  % \corauthref{cor1},
  \\
  Elisa Mastrogiacomo\thanks{Universit\'a degli studi di Milano Bicocca,
    Italia. {\tt elisa.mastrogiacomo@unimib.it}}
   \\
 Boubaker Smii${}^\sharp{}$\thanks{${}^\sharp{}$King Fahd University of
Petroleum and Minerals, Dept. Math. and Stat., Dhahran 31261, Saudi
Arabia. \hspace*{5mm}{\tt boubaker@kfupm.edu.sa}}}
\date{}
\maketitle

  \begin{abstract}
    We study a reaction-diffusion evolution equation perturbed by a space-time L\'evy noise. The
associated Kolmogorov operator is the sum of the infinitesimal generator of a $C_0$-semigroup
of strictly negative type acting in a Hilbert space and a nonlinear term which has at most
polynomial growth, is non necessarily Lipschitz and is such that the whole system is dissipative.

The corresponding It\^o stochastic equation describes a process on a
Hilbert space with dissi- pative nonlinear, non globally Lipschitz
drift and a L\'evy noise. Under smoothness assumptions on the
non-linearity, asymptotics to all orders in a small parameter in
front of the noise are given, with detailed estimates on the
remainders.

 Applications
to nonlinear SPDEs with a linear term in the drift given by a Laplacian in a bounded domain
are included. As a particular case we provide the small noise asymptotic expansions for the
SPDE equations of FitzHugh Nagumo type in neurobiology with external impulsive noise.

  \end{abstract}

\begin{center}\begin{minipage}{.8\textwidth}{%
\small {{\it Key words:} SPDE's equations, dissipative systems, L\'evy processes, stochastic convolution with L\'evy processes, asymptotic expansions, polynomially bounded non linearity, stochastic FitzHugh-Nagumo system}  %\begin{keyword}
 {\par\leavevmode\hbox {\it MSC(2010):}} 35C2060H15, 60651
    {\par\leavevmode\hbox {\it 1991 MSC:}} Primary 35K57; Secondary 92B20, 35R60, 35C20 }
\end{minipage}\end{center}
% \end{keyword}
%\end{frontmatter}
\date{\null}
% main text

%\setlength{\textheight}{7.7truein}    %FOR 2ND PAGE ONWARDS
%
%%%%%%%%%%%%%%%%%%%%%%%%%%%%%%%%%%%%%%%%%%%
%\author[S. Albeverio ] {  \it Sergio Albeverio \\ Dept. Appl. Mathematics, University of Bonn, HCM; SFB611, BiBoS, IZKS \\
%albeverio@uni-bonn.de}
%
%
%
%\author[L. Di Persio]{\medskip \\ }
%
%
%
%\author[E. Mastrogiacomo]{\medskip \\ \it Elisa Mastrogiacomo\\ Politecnico di Milano \\ elisa.mastrogiacomo@polimi.it}
%
%
%\title[]{Invariant measures for stochastic PDEs with dissipative polynomially bounded nonlinearity and applications on neural networks.}
%
%
%\date{}
%
%\begin{abstract}
%
%
%\end{abstract}
%
% \subjclass[2010]{Primary 35K57, 35R60, 35C20; Secondary 92B20}
% \keywords{Reaction diffusion equations, dissipative systems, asymptotic expansions, polynomially bounded non linearity, stochastic FitzHugh-Nagumo system}
%\maketitle

\section{Introduction}

In many areas of investigations, in natural as well as technical and
socio-economical sciences, a description of phenomena in terms of
(partial) differential equations (PDEs) is quite natural and has
received a lot of attention, also in recent years. However the
necessity of taking care of stochastic (or random) influences on
systems primarly described by (P)DEs in particular through
stochastic (P)DEs, S(P)DEs for short, has also came  to the
forefront of research, see, e.g. \cite{Wal, Kall},
\cite{DPZVerde},\cite{DPZRosso}, \cite{PeZa,Mumf, GaMa, Holden et
al}
 In the present paper we concentrate on PDE's perturbed by a space-time noise of the
additive type. Such SPDE's have been studied extensively
particularly in the case of Gaussian noises and they have found
applications in several areas, from physics to biology and financial
mathematics, see e.g. \cite{Wal}, \cite{DPZVerde}, \cite{DPZRosso},
\cite{GaMa, Holden et al, AlMaLy, Carm, AlRo,PRRO}. The extension to
the treatment of additive L\'evy type noises (which are more general
in the sense that random variables with L\'evy distributions extend
Gaussian random variables) is relatively more recent, see e.g.
\cite{AlWuZh, PeZa}. A natural question which arises in such
extensions from a deterministic description of phenomena to a
stochastic one, is in which sense one can recover the deterministic
description by "switching off" the noise and possibly obtain "small
noise expansions" around the limit. In the case of SDE's (as
opposite to SPDE's) this is a rather well studied problem,
especially in the case of Gaussian noises and it has also relations
with the study of the classical limit from quantum mechanics(see
e.g. \cite{Wa, TuEsp,IkWa, AMa, ASK, AL, Marc2, Ma, MaTa, Si, IK07,
InKa}). The infinite dimensional case and the case of SPDEs is less
studied, even in the case of Gaussian noises, see however \cite{CeF,
AlMaLy, ARoSk, RT, Ma,Marc2}. Concerning applications, the case of
stochastic perturbations of the FitzHugh-Nagumo equations of
neurobiology and its relations with the classical, deterministic
FitzHugh Nagumo equations is particularly interesting, due to the
fact that those equations, and their extensions to the case where
the underlying euclidean domain in space is replaced by a network,
are extensively used in neurobiology, see e.g. \cite{CaMu, AlDiP,
Tu, Wal}. In two recent papers \cite{AlDiPMa}, \cite{ADPM} a
systematic study of SPDE's with additive Gaussian noise which
includes in particular the above stochastic FitzHugh-Nagumo
equations, has been given, together with a detailed study of the
corresponding diffusion expansion around the deterministic limit.
One basic difficulty which had to be overcome (besides the infinite
dimensionality of the stochastic process involved) consisted in the
non global Lipschitz character of the nonlinear terms. A global
Lipschitz condition would in fact exclude the interesting case of
the FitzHugh-Nagumo equations; similarly other interesting equations
like those arising in stochastic quantization \cite{ALKR},
hydrodynamics \cite{AFER} or solid state physics ( e.g, Allen-Cahn
equations), would be excluded. Despite the interest of modeling the
noise in such systems through a L\'evy-type one instead of a
Gaussian one, which has motivations in all the areas which have been
mentioned, apparently a corresponding study of asymptotic expansions
around the underlying deterministic systems has yet to be performed.
We shall here adopt the method used in \cite{AlDiPMa} to cope with
this case. The adaptation involves, in particular, using methods
developed by \cite{PeZa} to handle stochastic convolutions in the
case of L\'evy noise. Let us note that our results seem to be new
even in the finite dimensional case, where small stochastic
perturbation expansions have also not been provided in details for
equations of the type we consider. Before we go over to describe the
contents of the present paper, let us mention that our study of
SPDE's with L\'evy noise can also be related to the study of certain
pseudo-differential equations with such noises which occur in
quantum field theory and statistical mechanics (see e.g \cite{AGoWu,
AlGYo}. Also relations to certain problems in the study of
statistics of processes described by L\'evy noises should be
mentioned \cite{GoSm, GST}.

   \section{Outline of the paper}

 Let us consider the following deterministic nonlinear evolution problem:
    \begin{equation}\label{eq:det}
    \begin{cases}
         d\phi(t)= [A\phi(t)+F(\phi(t))]dt  , \quad t \in [0,+\infty) \\
         \phi(0) = u^0, \quad u^0 \in H\:,
    \end{cases}
    \end{equation}
    where
    $A$ is a linear operator on a separable Hilbert space $H$ which generates a $C_0$-semigroup of strict negative type.
    The term $F$ is a \textit{smooth} nonlinear, quasi-$m$-dissipative mapping from the domain $D(F)\subset  H$ (dense in $H$) with values in
    $H$ (this means that there exists $ \omega \in \mathbb{R}$ such that $(F-\omega I )$ is $m$-dissipative in the sense of \cite[p. 73]{DPZVerde}),
    with (at most) polynomial growth at infinity and satisfying some further assumptions which will be specified in Hypothesis \ref{hp:A+F} below.)
   % while $D(F)$ is the domain of $F$, assumed to be dense in $H$.
    Existence and uniqueness of solutions for equation \eqref{eq:det} is discussed in Proposition \ref{prop:MildDeterministica} below.

Our aim is to study a stochastic (white noise) perturbation of \eqref{eq:det} and to write its (unique) solution as an expansion in powers of a parameter $\ep>0$, which controls the strength of the noise, as  $\varepsilon$ goes to zero. More precisely, we are concerned with the following stochastic  Cauchy problem on the  Hilbert space $H$:
    \begin{equation}\label{eq:eps}
    \begin{cases}
         du(t)= [Au(t)+F(u(t))]dt + \varepsilon \sqrt{Q}dL(t) , \quad t \in [0,+\infty) \\
         u(0) = u^0, \quad u^0 \in K \:,
    \end{cases}
    \end{equation}
    where $A$ and $F$ are as described above, $L$ is a mean square integrable L\'evy process taking values in a Hilbert space $U$, $Q$ is a positive trace
    class linear operator from $H$ to $H$ and $\ep>0$ is the parameter which determines the magnitude of the stochastic perturbation. The initial datum
    $u^0$ takes values in a continuously embedded Banach space $K$ of $H$.
  A unique solution of the problem \eqref{eq:eps} can be shown to exist  exploiting  as in \cite{BoMa} results on stochastic differential equations
  (contained, e.g., in  \cite{DPZRosso, DPZVerde}).
      Our purpose is to show that the solution of the equation \eqref{eq:eps}, which will be denoted by $u=u(t),  t \in [0,+\infty)$, can be written as
        $$
         u(t) = \phi(t) + \ep u_1(t) + \dots + \ep^n u_n(t) + R_n(t, \ep) \:,
        $$
        where  $n$ depends on the differentiability order of $F$. The function $\phi(t)$ solves the associated deterministic problem \eqref{eq:det},
        $u_1(t)$ is the stochastic process which solves the following linear stochastic (non-autonomous) equation
        \begin{equation} \label{eq:u1}
         \begin{cases}
         du_1(t)= [Au_1(t)+\nabla F(\phi(t)) [u_1(t)]]dt +  \sqrt{Q}dL(t) , \quad t \in [0,+\infty) \\
         u_1(0) = 0\:,
    \end{cases}
        \end{equation}
        while for each $k=2,\ldots,n \; ,$ $u_k(t)$ solves the following non-homogeneous linear differential equation with stochastic coefficients %(linear in $u_k$)
        \begin{equation}\label{eq:uk}
         \begin{cases}
         du_k(t) = \left[A u_k(t) + \nabla F (\phi (t) )[u_k(t)]\right]  dt + \Phi_k(t) dt, \\
         u_k(0) = 0 \:.
         \end{cases}
        \end{equation}
     $\Phi_k(t)$ is a stochastic process which depends on $u_1(t),\ldots,u_{k-1}(t)$ and the Fr\'echet derivatives of $F$ up to order $k$,
     see Section \ref{sec:2} for details.

Let us shortly describe the content of the different sections of the
present paper. In Section 3 we set the basic assumptions needed to
perform the construction of solutions and their asymptotic
expansion. In section 4 we describe the mild solutions to SDE's
driven by L\'evy processes on Hilbert spaces, basically following
the setting of \cite{PeZa}. Since our expansion will be around
solutions of the corresponding deterministic equations, we start by
presenting results on the latter equations (Subsection 4.1). In the
Subsection 4.2 we present the setting for the stochastic
perturbation, first describing the noise. In Section 5 we describe
the basic assumptions on the nonlinear term and provide its Taylor
expansion. Section 6 contains the main results, in particular the
construction of the expansion, the proof of its asymptotic character
and of detailed estimates on remainders, to any order. We close with
an application to the case of a FitzHugh Nagumo equation on a
network.

%as in Da Prato \& Zabczyk \cite{DPZverde}, we start by presenting our notation and by making some assumptions on the operators $A$
\section{Assumptions and Basic Estimates} \label{sec:1}
Before recalling some known results on problems of the types
\eqref{eq:det}, \eqref{eq:eps}, \eqref{eq:u1} and \eqref{eq:uk}, we
begin by presenting our notation and assumptions. We are concerned
with a real separable Hilbert space, with the inner product $\langle
\cdot,\cdot\rangle$. Moreover, in what follows, $(B, |\cdot|)$ is a
reflexive Banach space continuously embedded into $H$ as a dense
Borel subset and $(K, |\cdot|)$ is a reflexive Banach space
continuously embedded in $B$.
%Thus we have
%\begin{align*}
%     K \subset B \subset H.
%\end{align*}
On $H$ there are given a linear operator $A: D(A)\subset H \to H$, a
nonlinear operator $F: D(F) \subset H \to H$ with dense domain in
$H$ and a bounded linear operator $Q$ on from $H$ to $H$. Moreover,
we are given a complete probability space $(\Omega,F, (F_t)_{t\geq
0}, \PP)$ which satisfies the usual conditions, i.e., the
probability space is complete, $F$ contains all $\PP$-null subsets
of sets in $F$ and the filtration $(F_t)_{t\geq 0}$ is right
continuous. Further, for any trace-class linear operator $Q$, we
will denote by ${\rm Tr}\, Q$ its trace; if $f$ is any  mapping on
$H$ which is Fr\'echet differentiable up to order $n$, $n\in
\mathbb{N}$,  we will denote by
      $f^{(i)}, \, i=1,\dots,n$  its $i$-th Fr\'echet derivative and by
$D(f^{(i)} )$ the corresponding domain (for a short survey on
Fr\'echet differentiable mappings we refer to Section 4). For any
$j\in \N$ and any vector space $X$, $L(X^j;X)$ denotes the space of
$j$-linear bounded mappings from
   $X^j$ into $X$ while the space of linear bounded mappings from $X$ into $L(X^j;X)$
is denoted by $L^j(X)$.
We denote by $|\cdot|_X$ the norm on $X$, by $\|\cdot\|_{L^j(X)}$ the norm of any $j$-linear operator on $X$ and by $\|\cdot\|_{HS}$ the Hilbert-Schmidt
norm of any linear operator
       on $X$.
Finally, for any $p\geq 1$, we will denote by $\mathcal{C}_{\mathcal{F}}([0,T]; L^p(\Omega;X))$ the space of $X$-valued, adapted mean square continuous processes
  $Y$ on the time interval $[0,T]$ such that the following norm is finite
$$
    |\!|\!| Y|\!|\!|= (\sup_{t \in[0,T]}\EE \left |Y(t)|_X^p \right)^{1/p}<\infty.
$$
    %\begin{notation}
    %\end{notation}
    \begin{hypothesis}\label{hp:A+F}
        \begin{enumerate}
        \item[]
             \item  The operator $A: D(A) \subset H \to H$ generates an analytic semigroup $(e^{tA})_{t\geq 0}$,
             on $H$ of strict negative type such that
             \begin{equation*}
                 \left\|e^{tA}\right\|_{L(H)} \leq e^{-\omega t}, \quad t\geq 0
             \end{equation*}
             with $\omega$ a strictly positive, real constant.% , and $ \mid Tr A^{-1} \mid < +\infty$,   (where $Tr$ stands for trace).
             % We also assume that there exists  an orthonormal basis $\left\{ e_k \right\}_{k \in \mathbb{N}}$ such that $A$ diagonalizes on it, that is there exists a sequence $ \left\{ \lambda_k \right\}_{k \in \mathbb{N}}$ of strictly positive increasing numbers such that: $\langle A e_k, e_k\rangle = -\lambda_k$ ($\langle \cdot , \cdot \rangle$ being the scalar product in $H$). We require in addition that: $\sum_{k \geq 1} \frac{1}{\lambda_k} < +\infty$.

             %\item There exist $\alpha \in (0,1/2)$ such that for arbitrary $T>0$
            % \textcolor{red}{(eventualmente sostituire $T$ con $\infty$)}
            % \begin{equation*}
            %     \int_0^T t^{-2\alpha} \left\|e^{tA}\sqrt{Q}\right\|_{HS}^2 ds <\infty.
            % \end{equation*}

%        Let assume that there exists a Banach space $K$, densely and continuously embedded $($as a Borel subset$)$ into $H$, endowed with the norm
%$|\cdot|_K$.
            Moreover, if $A_B$ denotes the part of $A$ in the reflexive Banach space $B$, that is
         $$
                D(A_B):=\left\{ x\in D(A)\cap B; A x\in B \right\}, \quad A_B x=Ax,
         $$
      then $A_B$ generates an analytic semigroup $($of negative type$)$ $e^{tA_B}, t\geq 0$ on $B$.
             \item The mapping $F: D(F) \subset H \to H$ is continuous,
             nonlinear, Fr\'echet differentiable up to order $n$ for some positive integer
             $n$ and quasi-$m$-dissipative, i.e., there exist $\eta>0$ such that
     $$
          \left\langle F(u)-F(v)-\eta(u-v),u-v\right\rangle < 0, \qquad for \ all \ u,v\in D(F).\\
     $$
     \item If $F^{(j)}_B$, $j=1,\dots,n$ denotes the part of $F^{(j)}$ in $B$, that is
          $$
                D(F^{(j)}_B):=\left\{ x\in D(F^{(j)})\cap K; F^{(j)}_B( x)\in B \right\}, \quad F^{(j)}_B (x)=F^{(j)}(x),
         $$
     then there exists a reflexive Banach space $K$ densely and continuously embedded in $B$ which makes the following assumptions satisfied:
          \begin{enumerate}
              \item there exists a positive real number
             $\gamma$ and a positive natural number $n $  such that:
             \begin{align*}
                 %&\left\langle F(u)-F(v)-\eta(u-v),u-v\right\rangle < 0\\
                 &\left| F_B(u) \right|_{B} \leq \gamma \left(1+\left|u\right|_{K}^{2n+1} \right),\quad u \in K,
             \end{align*}
    \item for some $n $ and any $u\in D(F_K^{(i)}), i=1,\dots,n$, there exist positive real constants $\gamma_i, \ i=1,\dots,n$ such that
\begin{align*}
     \|F^{(i)}_B(u) \|_{L^{j}(B)} \leq \gamma_i(1+|u|_K^{2n+1-i}) \:,\quad \textrm{with $n$ as in $(${\rm iii}$)$, $u\in K$}
\end{align*}
      \end{enumerate}
        %     \begin{comment}
        %      Moreover we assume that $F$ has the following expansion for any
        %     $\phi,u_1,\dots,u_n \in H$: define $h_n=\ep u_1+\dots+\ep^n u_n$; then
        %     \begin{equation*}
        %        F(\phi+h) = F(\phi)+(\nabla F(\phi)[h])_*+ (\nabla^2F(\phi)[h,h])_{*} +
        %        \dots + (\nabla^{(n)}F(\phi)[h,\dots,h])_{*}+R_F(\phi,\dots,u_n,\ep)
        %     \end{equation*}
        %     where the subscript $*$ means keep the terms only up to order $\ep^n$, $R_F$ verifies
        %     \begin{equation*}
        %         R_F(\phi,\dots,u_1) \leq C \left(\left|\phi\right|,\dots,\left|u_n\right|\right)\ep^{n+1}
        %     \end{equation*}
        %     and $C$ is a function of $\left|\phi\right|,\dots,\left|u_n\right|$ with subpolynomial growth.
        %     \end{comment}

        \item The constants $\omega,\eta$ satisfy the inequality $\omega -\eta >0$; this implies that
         the term $A+F$ is $m$-dissipative in the sense of \cite{DPZRosso}, \cite[p. 73]{DPZVerde}.
         \item The term $L$ is a L\'evy process $($for example in the sense of  \cite{AW, PeZa}$)$ on some Hilbert space $U$; moreover we assume that
    \begin{align*}
     \int_U |y|^m \nu(\dd y) <\infty,
\end{align*}
   for all $m\in \N$, where $\nu$ is the jump intensity measure introduced in section
 (\ref{ABOE}).
         \item $Q$ is a positive linear bounded operator on $H$ of trace class, that is ${\rm Tr}\,Q<\infty$.
        \end{enumerate}
    \end{hypothesis}

  \begin{example} \label{ex:F} \rm

Let us give an example of a mapping $F$  satisfying the above
hypothesis (in view of the application to stochastic neuronal
models, which we will present in section 6).
   Let $H = L^2(\Lambda)$ with $\Lambda \subset \mathbb{R}^n$, bounded and open; set
  $B:=L^{2(2n+1)}(\Lambda)$, $K:=L^{2(2n+1)^2}(\Lambda)$
%(with $\bar{\Lambda}$ the closure of $\Lambda$ in $\R^n$)
and let $F$ be a multinomial of odd degree $2n+1$, $n \in \mathbb{N}$, i.e. a mapping of the form
$
    F(u)=g_{2n+1}(u)
$, where $g_{2n+1}(u)$, $u \in H$, is a polynomial of degree $2n+1$,  that is, $ g_{2n+1}(u)=a_0 +a_1 u+ \dots +a_{2n+1} u^{2n+1}$, with $a_i \in \mathbb{R}$, $i=0,\ldots,2n+1$. Then
it is easy to prove that
$D(F) = L^{2(2n+1)}(\Lambda) \subsetneq L^2(\Lambda),n>0$, $D(F)=L^{2}(\Lambda)=H,n=0$ and (by using the H\"older inequality) $D(F^{(i)})=L^{2(2n+1-i)}(\Lambda)$. Moreover, it turns out that, for any $u\in K$, $F^{(i)}(u)$ can be identified with the element $g_m^{(i)}(u)$ (both in $D(F)$ and $K$).
Consequently,
\begin{align*}
      |F(u)|_B &= \left(\int_{\Lambda}|g_{2n+1}(u(\xi))|^{2(2n+1)} \dd \,\xi\right)^{1/(2(2n+1))} \\
       %&=  \left(\int_{\Lambda} |g_{2n+1}(u(\xi))|^{2(2n+1)}\dd\, \xi\\
       &  \leq C_{2n+1} \left(1+ \int_{\Lambda}|u(\xi)|^{2(2n+1)^2} \dd \,\xi\right)^{1/(2(2n+1))}  \\
      &= C_{2n+1} (1+|u|_K^{2n+1})
\end{align*}
and, similarly,
\begin{align*}
      |\nabla^{(j)} F(u)|_{L^j(K;B)} &  \leq C_{2n+1-i} (1+|u|_K^{{2n+1-i}}) \\
      &= C_{2n-i} (1+|u|_K^{2n+1-j}), \qquad j=0,1,\dots,m.
\end{align*}
Hence $F$ satisfies Hypothesis \ref{hp:A+F} (ii), (iii). Further, in
the case $g_3(u)=-u(u-1)(u-\xi), 0<\xi<1$ the corresponding mapping
$F$ coincides with the non linear term of the first equation in the
FitzHugh-Nagumo system (see Example \ref{Remark:FHN} below).
   %%%%%%%%%%%%%%%%%%%%%%%%%%%%%%%%%%%%%%%%%%%%%%%%%%%%%%%%%%%%%%%%%%%%%%%%%%%%%%
   %%%%%%%%%%%%%%%%%% NUOVO REMARK STIME NABLA F  %%%%%%%%%%%%%%%%%%%%%%%%%%%%%%%
   %%%%%%%%%%%%%%%%%%%%%%%%%%%%%%%%%%%%%%%%%%%%%%%%%%%%%%%%%%%%%%%%%%%%%%%%%%%%%%

    \end{example}
    %More precisely, the following hypothesis
    %will be imposed everywhere in what follows.

    %The main result of this section is
    %\begin{theorem}
    %    Let $u(t)$ be the solution of \eqref{eq:eps}. Under the hypothesis
    %    ... $u$ can be represented in the following form:
    %    \begin{equation*}
    %        u(t)=\phi(t)+ \varepsilon u_1(t)+\varepsilon^2u_2(t)+ \cdots + \varepsilon^n u_n(t) + R_n(t,\varepsilon)
    %    \end{equation*}
    %    where
    %    \begin{equation*}
    %        \lim_{\ep \to 0} \dfrac{\EE\left[\sup_{t\in [0,+\infty)} \left|R_n(t,\varepsilon)\right|^2 \right]}{\ep^n} =0.
    %    \end{equation*}
    %\end{theorem}

\section{Mild solutions to SDE's driven by L\'{e}vy on Hilbert spaces}
   \noindent
In this section we basically use the setting of \cite{PeZa}.

% Kapitel 2.1: The deterministic case
\subsection{The deterministic case}
\label{sec:section1}

Let $A_0$ be a densely defined linear operator on a Banach space $B$, with domain $D(A_0)$. Let assume that the differential equation
\begin{equation}
    \left\{
        \begin{aligned}
            \frac{\mathrm{d}y}{\mathrm{d}t} &= A_0y \\
            y(0) &= y_0 \in D(A_0)
        \end{aligned}
    \right.
    \label{eq:equation1}
\end{equation}
has a unique solution $y(t)$, $t\geq0$, $y(t)\in B$. The equation being linear, we have
\begin{equation}
    y(t) = S(t)y_0\text{,} \quad t\geq0\text{,}
    \nonumber
\end{equation}
with $S(t)$ a linear operator from $D(A_0)$ into $B$. If for each
$t\geq0$, $S(t)$ has a \underline{continuous} extension to all of
$B$, and for each $z\in B$, $t\to S(t)z$ is continuous, then one
says that the Cauchy problem (\ref{eq:equation1}) is well posed.
$t\to S(t)z$, defined then for all $z\in B$, is called a generalized
solution to (\ref{eq:equation1}). One has: that
$\bigl(S(t)\bigr)_{t\geq0}$ is a $C_0$-semigroup:
\begin{enumerate}
    \item $S(0)=\mathbb{1}$, $S(t)S(s)=S(t+s)$, $t,s\geq0$,
    \item $\left| S(t)z-z \right|_{B} \to 0$ as $t\downarrow 0$, for every $z\in B$.
\end{enumerate}
Let $D(A)$ be the definition domain of the generator $A$ of $S(t)$. We have $D(A)\supset D(A_0)$ and $A$ is an extension of $A_0$.
 Moreover, see,  e.g. (\cite{PeZa}, Theorem 9.2):
\begin{enumerate}
    \item $\left| S(t)z \right|_{B} \le \mathrm{e}^{\omega t}M\left| z \right|_{B}$, for some $\omega,M>0 \ \forall z\in B$, $\forall t\geq0$,
    \item $A$ is closed and for any $z\in D(A)$, $t>0$ one has $S(t)z\in D(A)$ and $\frac{\mathrm{d}}{\mathrm{d}t}S(t)z=AS(t)z=S(t)Az$.
     In particular for $z=y_0\in D(A)$, $t\to S(t)y_0$ solves (\ref{eq:equation1}) with $A$ replacing $A_0$.
\end{enumerate}
Now, let $H$ be a Hilbert space such that $B\subset H$, with a dense, continuous embedding, $B$ being a Borel subset of $H$. Let $\psi(t)$, $t\geq0$ be $H$-valued and continuously differentiable. Then the \grqq variation of constants formula\glqq
\begin{equation}
    y(t) = S(t)y_0 + \int\limits_{0}^{t}S(t-s)\psi(s)\,\mathrm{d}s\text{,} \quad t\geq0
    \label{eqn:equation2}
\end{equation}
solves
\begin{equation}
    \left\{
        \begin{aligned}
            \frac{\mathrm{d}y}{\mathrm{d}t}(t) &= Ay(t) + \psi(t) \\
            y(0) &= y_0 \in H\text{.}
        \end{aligned}
    \right.
\end{equation}
In general, whenever the integral in (\ref{eqn:equation2}) has a meaning for a given $y_0$ in $H$, one says that (\ref{eqn:equation2}) is a
 \underline{mild solution} of
\begin{equation}
    \left\{
        \begin{aligned}
            \frac{\mathrm{d}y(t)}{\mathrm{d}t} &= Ay(t) + \psi(t) \\
            y(0) &= y_0\text{.}
        \end{aligned}
    \right.
\end{equation}

The formal definition of mild solution is given below.
  \begin{definition}\label{def:det}
       Let $y_0 \in H$; we say that the function $\phi: \, [0,\infty) \to H$ is a mild solution of equation \eqref{eq:det}
       if it is continuous $($in $t)$, with values in $H$ and it satisfies:
       \begin{equation} \label{mildsolution}
           \phi(t)=e^{tA}y_0+\int_0^t e^{(t-s)A}\psi(s) \, \dd s, \quad t\in[0,+\infty),
       \end{equation}
    with the integral existing in the sense of Bochner integrals on Hilbert spaces.
     \end{definition}

In the case of $\psi$ being substituted by a mapping $F: D(F)
\subset H \to H$ satisfying the assumptions given in Hypothesis
$\ref{hp:A+F}$ we have the following result.

    \begin{proposition}\label{prop:MildDeterministica}
        Under Hypothesis $\ref{hp:A+F}$ there exists a unique mild solution $\phi=\phi(t), t \in [0,\infty)$
        of the deterministic problem
\begin{equation}
    \left\{
        \begin{aligned}
            \frac{\mathrm{d}y}{\mathrm{d}t} &= A_0y + F(y) \\
            y(0) &= y_0 \in D(A_0)
        \end{aligned}
    \right.
    %\label{eq:equation1}
\end{equation} such that
        \begin{equation}\label{stimaLunardi}
            |\phi(t)|_H \leq e^{-2(\omega-\eta) t}|u^0|_H, \quad t\geq 0.
        \end{equation}
    \end{proposition}
    \begin{proof}

    The proof of existence and uniqueness can be found,e.g in \cite[Theorem 7.13, p. 203]{DPZRosso},
     while estimate \eqref{stimaLunardi} is a direct consequence of the application of Gronwall's lemma to the following inequality
    \begin{align*}
     \frac{d}{dt} | \phi(t) |_{H}^2 &= 2 \langle A \phi(t), \phi(t) \rangle dt + 2\langle F(\phi(t)),\phi(t) \rangle \\
   &\leq - 2(\omega-\eta)|\phi(t)|_H^2.
    \end{align*}
    \end{proof}

   \begin{remark}
         It can be shown that, under Hypothesis $\ref{hp:A+F}$, there exists a $K$-continuous version of the unique solution of equation \eqref{mildsolution} such that, for any $T>0$, $p\geq 1$
     \begin{align*}
         \sup_{t\in [0,T]} |\phi(t)|_K^p <\infty.
     \end{align*}
      $($see \cite[Section 5.5.2, Proposition 5.5.6]{DPZVerde}$)$.
        Hence, in the following, by $\phi$ we will always understand this $K$-valued version of the solution of \eqref{eq:det}.
    \end{remark}

% Kapitel 2.2: The stochastically perturbed case
\subsection{The stochastically perturbed case}\label{ABOE}
 Let $G$ be a linear operator from a Hilbert space $U$ into a
Hilbert space $H$. Let $S(t)$ be a $C_0$-semigroup on the Hilbert
space $H$.
 Assume the generator $\left(A,D(A)\right)$ of $S(t)$ in $H$ is almost $m$-dissipative (i.e.\ $[(\lambda\mathbb{1}-A)+\eta]H=H$ for any
 $\lambda>0$ and some $\eta\in\mathbb{R}$: [\cite{PeZa}, p. 180]; this is equivalent to quasi m-dissipative in the sense of \cite[p. 73]{DPZVerde}.) Assume $B \subset H$ as in \eqref{sec:section1} and that the restriction $A_{B}$ of $A$ to
 $B$ is also almost $m$-dissipative. Let $L$ be a square-integrable mean zero L\'{e}vy process taking values in a Hilbert space $K$.
 I.e.\ $L=\bigl(L(t)\bigr)_{t\geq0}$ takes values in a Hilbert space $K$, has independent, stationary, increments, one has $L(0)=0$, and $L(t)$ is
 stochastically continuous (see \cite{PeZa}, Def. 4.1, p.\ 38). Let $Q$ be the covariance of $L$. Then $Q^{\frac{1}{2}}(K)$ is the reproducing kernel Hilbert space
 (RKHS) of $L$, assume $Q^{\frac{1}{2}}(\textcolor{red}{K})$ is embedded into $U$. \\
We recall the following basic notions and results.
\begin{definition}
    Let $\nu$ be a finite measure on a Hilbert space $U$ such that $\nu(\left\{0\right\})=0$.
A compound Poisson process with L\'evy measure (also called jump
intensity measure) $\mu$ is a c$\grave{a}$dl$\grave{a}$g L\'evy
process $L$ satisfying
\begin{align*}
   P(L(t)\in \Gamma)= e^{-\nu(U)t}\sum_{k=0}^\infty \frac{t^k}{k!} \nu^{*k}(\Gamma), \qquad t\geq 0, \Gamma\,\in\, \mathcal{B}(U).
\end{align*}
$\mathcal{B}(U)$ being the $\sigma-$algebra of Borel subsets of $U.$
\end{definition}
Given a Borel set $I$ separated from $0$, write
\begin{align*}
  \pi_I(t)=\sum_{s\leq t} \chi_I(\Delta L(s)), \qquad t\geq 0.
\end{align*}
The c$\grave{a}$dl$\grave{a}$g property of $L$ implies that $\pi_I$
is $\mathbb{Z}_+$-valued. We notice that it is a L\'evy process with
jumps of size
 $1$ and thus, a Poisson process (see [\cite{PeZa}, Proposition 4.9 (iv)] for more details.) We also have that $\mathbb{E} \pi_I(t)= t \mathbb{E} \pi_I(t)=t\nu(I)$,
  where $\nu$ is a measure that is finite on sets separated from $0$. We shall write
\begin{align*}
    L_I(t)=\sum_{s\leq t} \chi_I(\Delta L(s)) \Delta L(s).
\end{align*}
Then $L_I$ is a well-defined L\'evy process. The theorem below
provides the corresponding L\'evy-Khinchine decomposition:
\begin{theorem}
    \begin{enumerate}
        \item If $\nu$ is a jump intensity measure corresponding to a L\'evy process then
\begin{align*}
    \int_U (|y|^2_U\wedge 1) \nu(\dd y) <\infty.
\end{align*}
   \item Every L\'evy process has the following representation:
   \begin{align*}
      L(t):= at + \sqrt{Q}W(t)+\sum_{k=1}^\infty \left(L_{I_k}(t)- t \int_{I_k} y \nu(\dd y)\right)+
    L_{I_0}(t),
\end{align*}
  where $I_0:=\left\{ x: |x|_U\geq r_0\right\}$, $I_k:= \left\{ x: r_k \leq |x|_U < r_{k-1}\right\}$, $(r_k)$ is an arbitrary sequence decreasing to $0$, $W$ is a Wiener process, all members of the representation are independent processes and the series converges $\mathbb{P}-a.s.$, uniformly on each bounded subinterval $[0,\infty)$.
    \end{enumerate}
\end{theorem}
In the following (see Hypothesis \ref{hp:A+F}), with no loss of
generality, we assume that
\begin{align}\label{eq:medianu}
\sum_{k=1}^\infty \int_{I_k} y \nu(\dd y) =0.
\end{align}
We also assume throughout that the L\'evy process is a pure jump process, i.e. $a=0$
and $Q=0$ and that
\begin{align}\label{eq:momenti}
   \int_U |y|^m \nu(\dd y) <\infty, \qquad for \ all \ m\in \mathbb{N},
\end{align}
which leads to the representation
\begin{align*}
    L(t)= \sum_{k=1}^\infty L_{I_k}(t)+ L_{I_0}(t),
\end{align*}
 in view of assumptions \eqref{eq:medianu} and \eqref{eq:momenti}.

Let $L_A(t)=\int\limits_{0}^{t}S(t-s)\sqrt{Q}\,\mathrm{d}L(s)$, $t\geq0$, be the L\'{e}vy Ornstein-Uhlenbeck process associated with $S,\sqrt{Q},L$,
assumed to exist and have a c\`{a}dl\`{a}g version in $B$ (the latter is satisfied if $B$ is a Hilbert space $K$ and $S(t)$ is a contraction on $K$),
 see e.g. (\cite{PeZa}, p.\ 155), or $S$ is analytic and $L$ takes values in $D\bigl((-A)^{\alpha}\bigr)$ for some $\alpha>0$; see, e.g.
  (\cite{PeZa}, p.\ 155). Assume $F$ is an operator on $H$ (possibly nonlinear, nor everywhere defined) satysfying Hypothesis \ref{hp:A+F}.
%Assume moreover that for all $T>0$ one has $P$-almost surely (a.s.) ($P$ being the underlying probability measure) $\int\limits_{0}^{T}\left|F\bigl(L_A(t)\bigr)\right|_{B}\mathrm{d}t<\infty$.

An adapted $B$-valued process $X$ is said to be a \underline{c\`{a}dl\`{a}g mild solution} to
\begin{equation}\label{eqn:equation3}
    \left\{
        \begin{aligned}
            \mathrm{d}X(t) &= AX(t)\,\mathrm{d}t + F\bigl(X(t)\bigr)\,\mathrm{d}t + \sqrt{Q}\,\mathrm{d}L(t) \\
            X(0) &= x \in D(F)
        \end{aligned}
    \right.
\end{equation}
if it is c\`{a}dl\`{a}g in $B$ and satisfies, $P$-a.s., the equation
\\
$X(t)=S(t)x+\int\limits_{0}^{t}S(t-s)F\bigl(X(s)\bigr)\,\mathrm{d}s+L_A(t)$,
$t\geq0$, with $X(s)\in D(F)$ for $s\geq0$ (\cite{PeZa}, p.\ 182).
The formal definition of mild solution for the stochastic problem
    \eqref{eqn:equation3} is given below; next we recall the definition of stochastic convolution and we list some of its properties.

    \begin{definition}\label{def:stoc-conv}
       Let $u^0\in K$. A predictable $H$-valued process $u:=(u(t))_{t\geq 0}$ is called
        a mild solution to the Cauchy problem \eqref{eq:eps} with initial condition $u^0 \in D(F)$ if for arbitrary $t\geq 0$ we have
       \begin{equation*}
          u(t)=e^{tA}u^0+\int_0^t e^{(t-s)A}F(u(s))ds + \ep \int_0^t e^{(t-s)A} \sqrt{Q}dL(s), \quad \textrm{$\PP$-a.s.}
       \end{equation*}
   $L_A(t) := \int_0^t e^{(t-s)A} \sqrt{Q}dL(s)$ is called a stochastic convolution and under our hypothesis it is a well defined mean square continuous
   $\mathcal{F}_t$-adapted process with values in $B$ and c$\grave{a}$dl$\grave{a}$g trajectories (see e.g., \cite{PeZa}, Proposition 9.28, p.\ 163).
    \end{definition}

   % We shall denote by $\Xi_{p,T}$ the space of all (equivalence classes) of predictable $H$-valued processes $v(t)$, $t\geq 0$, such that
   % \begin{equation*}
   %     \|v(t)\|_{p,T}=\left(\EE \sup_{t \in [0,T]} |v(t)|_H^2\right)^{1/2} < \infty.
   % \end{equation*}

   The first integral on the right hand side is defined pathwise in the Bochner sense, $\mathbb{P}$-almost surely.

   For further use, in the following we introduce some additional condition on the stochastic convolution:
    \begin{hypothesis}\label{prop:StochasticConvolution}
        The stochastic convolution $L_A(t),t\geq 0$ introduced in Definition $\ref{def:stoc-conv}$, admits a $K$-valued version such that,
        for any $T>0$, it satisfies the following estimate
         {\small
         \begin{equation}
          \EE\left(\sup_{t \in [0,T]} | L_A(t) |_K^{m}
%2n+1}
        \right) \leq C_T
         \end{equation}
         }
         for every $m\in \bN$ and some positive constant $C_T$ $($possibly depending on $T)$.
    \end{hypothesis}

    \begin{example}\rm
          Let us give an example for the setting $(H,B,K,L,A,Q)$ where $L_A$ is well-defined and Hypothesis \ref{prop:StochasticConvolution} is satisfied.
           This example is related to the application to the stochastic FitzHugh-Nagumo model which we discuss in Example \ref{Remark:FHN}.
Let $H,B,K$ be as in Example \ref{ex:F}. Let $A=\Delta$ be the
Laplacian in $L^2(\Lambda)$ with Neumann boundary conditions on the
boundary $\partial\,\Lambda$ of the bounded open subset $\Lambda$ of
$\R^n$.
 Let $Q$ be a bounded trace class operator commuting with
  $A$ and $L$ be a L\'evy process such that the corresponding measure $\nu$ satisfies
\begin{align*}
     \int_{L^2(\Lambda)} |x|^m_{W^{\beta,2(2n+1)}} \nu(\dd x) <\infty \quad for \ all \ m\in \bN,
\end{align*}
where $W^{\beta,2(2n+1)}$ is a fractional Sobolev space with given $\beta >0$. Finally, let
  $(A_{18},D(A_{18}))$ denote the generator of the heat semigroup with Neumann boundary conditions on $L^{18}(\Lambda)$.
  By {\bf \cite[Appendix ]{DPZRosso}} $L_A(t)\in D((-A_{2(2n+1)}^\gamma)$, $\gamma>0$; in particular $L_A(t)\in K$, $L_A$ being in addition a L\'evy process.
  This implies the bound in
    Hypothesis \ref{prop:StochasticConvolution}.
    \end{example}

% Theorem

      The next result concerns the existence and uniqueness of solutions for the stochastically perturbed problem. Moreover,
    we shall use Hypothesis \ref{prop:StochasticConvolution} above concerning the Ornstein-Uhlenbeck process associated with $e^{tA},\sqrt{Q}$ and $L$ in order
    to prove a useful estimate on the solution.

\begin{thm}
       Assume that $A$ and $F$ satisfy Hypothesis $\ref{hp:A+F}$. Assume that $A$ and $Q$ satisfy Hypothesis $\ref{prop:StochasticConvolution}$.
       Then there exists a  unique c\`{a}dl\`{a}g mild solution of (\ref{eqn:equation3}) for any $x\in B$. For each $x\in H$ there exists a unique generalized
       solution for (\ref{eqn:equation3}) (in the sense that $\exists(X_n)_{n\in\mathbb{N}}$, $X_n\in B$, unique c\`{a}dl\`{a}g mild solutions of \eqref{eqn:equation3} with $X_n(0)=x$
        s.t.\ $\left| X_n(t)-X(t) \right|_{H}\to0$ uniformly on each bounded interval). Moreover (\ref{eqn:equation3}) defines Feller families on $B$ and on $H$
         (in the sense that the Markov semigroup $P(t)$ associated with $X(t)$ maps for any $t\geq0$, $C_{b}(H)$ into $C_b(H)$ and $C_b(B)$ into $C_b(B)$).

Moreover, the solution $X$ to (\ref{eqn:equation3}) belongs to the
space  $\mathcal{L}^p(\Omega;C([0,T];H))$,  i.e., is such that
         \begin{equation}\label{stima:u}
             \EE\left( \sup_{t \in [0,T]} \left|X(t)\right|^p_H \right)< +\infty,
         \end{equation}
         for any $p \in[2, \infty)$.
\end{thm}

%    The paper is organized as follows.
%    In Section \ref{sec:1} we recall standard results for the solution of equations of types \eqref{eq:det}, \eqref{eq:eps}, \eqref{eq:u1} and \eqref{eq:uk}.
%    Section \ref{sec:2} is devoted to the study of some properties of the nonlinear term $F$, in particular  the $n$-th remainder of its  Taylor expansion. Section \ref{sec:3} is concerned with the proof of the main result on the asymptotic expansion in powers of $\ep$ of the solution of the stochastic equation \eqref{eq:eps}, with explicit coefficients and remainders, and estimates thereof.
%%    which gives the asymptotic behaviour of the Taylor remainder  $R_n(t,\varepsilon)$ derived in previous section.
%We conclude with some remarks on applications of the results, in particular concerning the stochastic FitzHugh-Nagumo equation.

%    \begin{proposition}\label{prop:MildStoc}
%       Then for any $u^0 \in D(F)$  and $T>0$, there
%         exists a unique mild solution $u=(u(t))_{0 \leq t \leq T} $  of the equation
%       \eqref{eq:eps} $($cf. Definition $\ref{def:stoc-conv})$ which belongs to the space
%   $\mathcal{L}^p(\Omega;C([0,T];H))$,  i.e., such that
%         \begin{equation}\label{stima:u}
%             \EE\left( \sup_{t \in [0,T]} \left|u(t)\right|^p_H \right)< +\infty,
%         \end{equation}
%         for any $p \in[2, \infty)$.
%    \end{proposition}

    \begin{proof}

    The first part of the result is proven in (\cite{PeZa}, Theorem 10.14).
    We only have to prove the estimate \eqref{stima:u}.
     Let $z(t) := X(t) - L_A(t)$;  then it is not difficult to show that $z(t)$ is the unique solution of the following deterministic equation:
     $$
     \begin{cases}
      z^\prime(t) = Az(t) + F(z(t)+L_A(t)) \\
      z(0) = u^0
     \end{cases}
     $$
     with $z^\prime (t) := \frac{d}{dt} z(t)$.

     With no loss of generality (because of inclusion results for $L^p$-spaces with respect to bounded measures) we can assume that $p=2a$, $a \in \N$.
     Now combining condition $(i)$ with $(i)$ in Hypothesis (\ref{hp:A+F}) and recalling
     Newton's binomial formula we have:
     \begin{equation}\label{eq:ItoFormula}
     \begin{aligned}
      \frac{d}{dt} |z(t)|_H^{2a} &=
       2a  \langle z^\prime(t),z(t) \rangle |z(t)|_H^{2a-2} = 2a \langle  Az(t) + F(z(t)+L_A(t)),z(t) \rangle |z(t)|_H^{2a-2} \\
       & \leq
       -2a \omega |z(t)|_H^{2a} + 2a \langle  F(z(t)+L_A(t)),z(t) \rangle |z(t)|_H^{2a-2} \\
       & \leq
        -2a (\omega -\eta)  |z(t)|_H^{2a} +2a | F(L_A(t)) |_H |z(t)|_H^{2a-1} \\
       &\leq
       -2a (\omega -\eta)  |z(t)|_H^{2a} + 2a \frac{C_a}{\xi}  | F(L_A(t)) |_H^{2a} +C_a  2a \xi | z(t) |^{2a}_H \:,
     \end{aligned}
     \end{equation}
     for some constant $C_a >0$ and a sufficiently small $\xi>0$ such that $-2a (\omega-\eta) +2a  \xi C_a  <0$.
     %  we made use of conditions (1) and (2) in Hypothesis \ref{hp:A+F} to obtain the above first resp. second inequality, while the third inequality is obtained by a Newton expansion of the $2a$ power of a sum.
     Applying the previous inequality and Gronwall's lemma we get:
     $$
       |z(t)|_H^{2a} \leq e^{(-2a (\omega-\eta) +\xi C_a 2a) t}  |u^0|_H^{2a} +
      \frac{ 2a C_a }{\xi}\int_0^t e^{-2a (\omega-\eta) (t-s)} |F(L_A(s))|_H^{2a} ds.
     $$
     Then  there exists a positive constant  $C$ such that:
     \begin{equation}\label{eq:Dis2N}
     |X(t)|_H^{2a} \leq C \left( e^{(-2a (\omega-\eta) +\xi C_a 2a) t}  |u^0|_H^{2a} +
       2a \int_0^t e^{-2a (\omega-\eta) (t-s)} |F(L_A(s))|_H^{2a} ds + |L_A(t)|_H^{2a} \right) .
     \end{equation}
     Since by condition $(iii)$ in Hypothesis \ref{hp:A+F}, the restriction of $F$ to $K$ has (at most)
     polynomial growth at infinity in the $K$-norm and, by the assumption on $L_A(t)$ made in Hypothesis \ref{prop:StochasticConvolution},
      $L_A$ takes value in $K$,  for any $a\in \N$ we have:
     $$
      |F(L_A(t))|_H^{2a} \leq
      C_{a,m} (1+|L_A(t)|_K^m)^{2a} \leq C_{a,m} (1+|L_A(t)|_K^{2am}),
     $$
     for some positive constant $C_{a,m}$ depending on $m$ and $a$.
     Moreover, we observe that, again by Hypothesis \ref{prop:StochasticConvolution}, it holds that
      $$
      \EE\left(\sup_{t \in [0,T]} |L_A(t)|_K^{2am}  \right)\leq C_{a,m,T}^\prime,
      $$ where $C_{a,m,T}^\prime$ is again a positive constant depending on $m$, $a$ and $T$; hence
     \begin{multline}\label{nonloso}
     %\begin{aligned}
      \EE \left[\sup_{t\in [0,T]} \int_0^t e^{-2a (\omega-\eta) (t-s)} |F(L_A(s))|_H^{2a} ds \right] \leq
      \tilde{C}  \EE  \left[\sup_{t \in [0,T]} \int_0^t e^{-2a (\omega-\eta) (t-s)}
      (1+|L_A(t)|_K^{2am})  ds \right] \\
      \leq \tilde{C}  \EE \left[ \sup_{t \in [0,T]} \int_0^t e^{-2a (\omega-\eta) (t-s)} ds +
       C_{a,m}^\prime \int_0^t e^{-2a (\omega-\eta)} ds \right]  \leq \bar{C}\:,
      % \end{aligned}
     \end{multline}
     for some positive constants $\widetilde{C}, \bar{C}$ depending on $a$, $m$ and $T$. Consequently,
putting together inequalities \eqref{eq:Dis2N}, \eqref{nonloso}, we obtain
   $$
       \EE\left( \sup_{t\in [0,T]}|X(t)|_H^{2a}\right) \leq C |u^0|_H^{2a} + \bar{\bar{C}} \:,
$$
for some positive constant $\bar{\bar{C}}$, so that the proposition follows.
%  Now the thesis follows taking the expectation in \eqref{eq:Dis2N} and
  %   \eqref{nonloso}.

    \end{proof}

    \section{Properties of the non-linear term $F$ and Taylor expansions} \label{sec:2}
    In this section we study the non-linear term $F$ in order to write its Taylor expansion around the solution $\phi(t)$ of \eqref{mildsolution} with respect to an increment given in terms of powers of $\varepsilon$.
    In order to do that we recall some basic properties of Fr\'echet differentiable functions.

    Let $U$ and $V$ be two real Banach spaces. For a mapping
$F:U\to V$ the   G\^ateaux differential at $u \in U$ in the direction
 $h\in U$ is defined as
  $$
  \nabla F(u)[h]=\lim_{s\to 0}\frac{F(u+sh)-F(u)}{s},
  $$
  whenever the limit exists in the topology of $V$ (see for example \cite[p. 12]{Lal}).

%%%%%%%%%%%%%%%%%%%%%%%%%%%%%%%%%%%%%%%%%%%%%%%%%%%%%%%%%%%%%%%%%%%%%%%%%%
%%%%%%%%%%%%%%%    INIZIO NUOVA PARTE  %%%%%%%%%%%%%%%%%%%%%%%%%%%%%
%%%%%%%%%%%%%%%%%%%%%%%%%%%%%%%%%%%%%%%%%%%%%%%%%%%%%%%%%%%%%%%%%%%%%%%%%%

We notice that if $\nabla F(u) [h]$ exists in a neighborhood of $u_0 \in U$ and is continuous in $u$ at $u_0$ and also continuous in $h$ at $h=0$, then $\nabla F (u) [h]$ is linear in $h$ (see for instance \cite[Problem 1.61, p 15]{Lal}).
If $\nabla F (u_0) [h]$ has this property for all $u_0 \in U_0 \subseteq U$ and all $h \in U$ we shall say that $F$ belongs to the space $G^1(U_0;V)$.
If $F$ is continuous from $U$ to $V$ and $F \in G^1(U_0;V)$ and one has $F(u+h) = F(u) + \nabla F(u)[h]+R(u,h)$, for any $u \in U_0$ with:

\begin{align}\label{eq:FrechetTaylor}
\lim_{\left| h \right|_U \rightarrow 0} \frac{\left| R(u,h) \right|_V }{\left| h \right|_U} =0
\end{align}
with $|\cdot|_V $ and  $| \cdot|_U$  denoting respectively the norm in $V$ and $U$, then the map $ h \rightarrow \nabla F (u) [h]$ is a bounded linear operator from $U_0$ to $V$, and $\nabla F(u)[h]$ is, by definition, the unique Fr\'echet differential of $F$ at $u \in U_0$ with increment $h \in U$.  The function $R(u,h)$ is called the remainder of this Fr\'echet differential, while
the operator sending $h$ into $\nabla F(u) [h]$ is then called the Fr\'echet derivative of $F$ at $u$ and is usually denoted by $F^\prime(u)$ (see for instance \cite[pp. 15-16, Problem 1.6.2 and Lemma 1.6.3]{Lal}).
We have then $\nabla F(u) [h] = F^\prime(u) \cdot h$, with the symbol $\cdot$ denoting the action of the linear bounded operator $F^\prime(u)$ on $h$.

The mapping $F^\prime(u)$ is also called the gradient of $F$ at $u$ (see for example \cite[p. 15]{Lal}) and it coincides with the G\^ateaux derivative of $F$ at $u$.
We shall denote by $\mathcal{F}^{(1)} (U_0,V)$ the subset of $G^1(U_0,V)$ such that the Fr\'echet derivative exists at any point of $U_0$.
Similarly we introduce the Fr\'echet derivative $F^{\prime\prime}(u)$ of $F^\prime$ at $u \in U$.
This is a bounded linear map from a subset $D(F^\prime)$ of $U$ into $L(U,V)$ ($L(U,V)$ being the space of bounded linear operators from $U$ to $V$). One has thus $F^{\prime\prime} \in L(U, L(U,V))$. If we choose $h,k \in U$ then $F^{\prime\prime}(u) \cdot k \in L(U,V)$ and $\left(F^{\prime\prime}(u) \cdot k\right) \cdot h \in V$. The latter is also written $F^{\prime\prime}(u) \;h\;k$ or $F^{\prime\prime}(u) [h,k]$. The mapping
$F^{\prime\prime}(u) [h,k]$ is bilinear in $h,k$, for any given $u \in D(F^{\prime\prime})$ and it can be identified with the G\^ateaux differential $\nabla^{(2)} F(u) [h,k]$ of $\nabla F(u)[h]$ in the direction $k$, the latter looked upon as a map from $U$ to $L(U,V)$.
Similarly one defines the $j$-th Fr\'echet derivative $F^{(j)}(u)$ and the $j$-th G\^ateaux derivative $\nabla F^{(j)}(u) [h_1, \ldots, h_j]$. The function
$F^{(j)}(u)$ acts $j$-linearly on $h_1,\ldots,h_j$ with $h_i \in U$ for any $i=1,\ldots,j$.
Let $U_0$ be an open subset of $U$ and consider the space $\mathcal{F}^{(j)}(U_0,V)$ of maps $F$ from $U$ to $V$ such that $F^{(j)}(u)$ exists at all $u \in U_0$ and is uniformly continuous on $U_0$. The following Taylor formula holds for any $u,h \in U$ for which $F(h)$ and $F(u+h)$ are well defined (i.e. $h$ and $u+h$ are elements of $D(F)$), and $j=1,\ldots,n+1$ with $u \in \cap_{j=1}^n \mathcal{F}^{(j)}(U_0,V)$:
\begin{equation}\label{eq:FrechetTaylorn}
 F(u+h) = F(u) + \nabla F (u) [h]+ \frac{1}{2} \nabla^{(2)}F(u) [h,h]+ \cdots + \frac{1}{n !} \nabla^{(n)} F(u) \underbrace{[h,\ldots,h]}_{\textrm{$n$-terms}}
 +R^{(n)}(u;h) \:,
\end{equation}
where $\left| R^{(n)}(u;h) \right|_U \leq C_{u,n} \cdot \left|h\right|_U^n$ for some
constant $C_{u,n}$ depending only on $u$ and $n$ (see for example \cite[Theorem X.1.2]{KolmogorovFomin}).

    Now let us consider the case $U = H$, with $H$ being the same Hilbert space appearing in
problem \eqref{eq:det}.
    %%%%%%%%%%%%%%%%
% For any $p\in [2,\infty)$ and $T>0$, we shall denote by $\Xi_{p,T}$ the space of all (equivalence classes) of predictable $H$-valued processes $v(t)$, $t\geq 0$, such that
  %  \begin{equation*} \label{eq:SpazioXiPT}
    %    \|v(t)\|_{p,T}=\left(\EE \sup_{t \in [0,T]} |v(t)|_H^p\right)^{1/p} < \infty.
     %   \end{equation*}
  %
    Let $F$ be as in Hypothesis \ref{hp:A+F} and set $U_0 = D(F)$. Let us define
     for $0< \varepsilon\leq 1$ the function $h(t)$, $t \geq 0 $:
    $$
     h(t) = \sum_{k=1}^n \varepsilon^k u_k(t)+r^{(n)}(t;\ep) \:,
    $$
    %which is an elment of $\Xi_{p,T}$.
    where the functions $u_k(t),k=1,\dots,n$ and $r^{(n)}(t;\ep)$ are $p$-mean integrable continuous stochastic processes with values in $H$, defined on the whole interval $[0,T]$ for $p \in [2,\infty)$. Moreover we suppose $r^{(n)}(\cdot;\ep) ={\bf o}(\ep^{n})$, i.e.,
  \begin{align*}
      \lim_{\ep \to 0}\EE\left[ \sup_{t\in [0,T]}\frac{|r^{(n)}(t;\ep)|^p}{\ep^n}\right] =0, \qquad for \ any \ T>0.
\end{align*}
    %$H-$valued process $\Xi_{p,T}$, $p$-mean integrable continuous stochastic processes  on the whole interval $[0,+\infty)$ for $p \in [2,\infty)$.
    Let $\phi$ be a $p$-mean integrable continuous stochastic process with values in the Banach space $K$. Then using the above Taylor formula we have
  \begin{equation}\label{eq:nablacasoconcreto}
\begin{aligned}
    F(\phi(t)+h(t))= F(\phi(t))+\nabla F (\phi(t)) [h(t)]+ \frac{1}{2}\nabla^{(2)} F[h(t),h(t)]+\cdots\\
    \qquad \qquad \qquad +
 \frac{1}{n !} \nabla^{(n)} F(u) \underbrace{[h(t),\ldots,h(t)]}_{\textrm{$n$-terms}}
 +R^{(n)}(\phi(t);h(t)) \:,
\end{aligned}
\end{equation}
   and, recalling that for any $j=1,\dots,n$, $\nabla^{(j)}F(\phi(t))$ is multilinear, we have
\begin{equation}\label{eq:oj}
   \begin{aligned}
    &\frac{1}{j!} \nabla^{(j)}F(\phi(t)) \underbrace{[h(t),\dots,h(t)]}_{\textrm{$j$-terms}}=\\
     &\qquad \qquad \frac{1}{j!}\sum_{k_1+\dots+k_j=j}^{nj} \ep^{k_1+\dots+k_j} \nabla^{(j)} F(\phi(t)) [u_{k_1}(t),\dots,u_{k_j}(t)]
      + {\bf o}_j(\ep^{nj})
\end{aligned}
\end{equation}
where ${\bf o}_j(\ep^{nj})$ is the contribution to the right member of the above equality coming from the term $r^{(n)}(t;\ep)$ and satisfies the estimate
\begin{align*}
     \lim_{\ep \to 0} \EE \left[\sup_{t\in [0,T]} \frac{|{\bf o}_j(\ep^{nj})|^p}{\ep^{nj}} \right]=0, \quad for \ any \ T >0.
\end{align*}
We notice that any derivative appearing in the  member on the right hand side of \eqref{eq:oj} is multiplied by the parameter $\ep$ raised to a power
between $j$ and $nj$.

   Taking into account the above equality we can rewrite \eqref{eq:nablacasoconcreto} as
    \begin{equation} \label{TaylorF}
      \begin{aligned}
       F(\phi(t)+h(t)) &= F(\phi(t)) +  \sum_{k=1}^n \ep^k \nabla F(\phi(t)) [u_k(t)]  \\
       &+ \sum_{j_1+j_2=2}^n
       \frac{\ep^{j_1+j_2}}{2!} \nabla^{(2)} F(\phi(t)) [u_{j_1}(t),u_{j_2}(t)]   + \cdots \\
       &+ \sum_{ j_1+\dots + j_k = k}^n
       \frac{\ep^{j_1+\dots+j_k}}{k!} \nabla^{(k)} F(\phi(t)) [u_{j_1}(t),\ldots,u_{j_k}(t)] +  \cdots \\
       &+ \frac{ \ep^n}{n!} \nabla^{(n)} F(\phi(t)) [u_{1}(t),\ldots,u_{1}(t)] + R_1^{(n)}(\phi(t);h(t),\ep)\:,
      \end{aligned}
    \end{equation}
    where the quantity $ R_1^{(n)}(\phi(t);h(t),\ep) $ is given in terms of the derivatives of $F$ with
the parameter $\ep$ raised to powers greater than $n$, in terms of the $n$-th remainder $R^{(n)} (\phi(t);h(t))$ in the Taylor expansion of the map $F$ (as stated in  equation \eqref{eq:FrechetTaylorn}) and in terms of the remainders ${\bf o}_j(\ep^{nj})$, $j=2,\dots, n$ introduced in \eqref{eq:oj}. Namely, we have:
% and in terms of the derivatives  depends on the , on the     solution  $\phi(t)$ of the deterministic equation \eqref{eq:det}  and on the terms containing  the parameter $\ep$  taken in powers larger of equal $n$, in fact we have:
     \begin{equation}\label{R1}
\begin{aligned}
       &R_1^{(n)} (\phi(t);h(t),\ep) = \sum_{j=2}^n\sum_{i_1+\cdots+i_j  = n+1}^{nj} \ep^{i_1+\dots+i_j}
       \frac{1}{j!} \nabla^{(j)} F(\phi(t)) [u_{i_1}(t), \dots , u_{i_j}(t)] \\
      &  \qquad \qquad \qquad \qquad \qquad + \sum_{j=2}^n {\bf o}_j(\ep^{nj})
       + R^{(n)}(\phi(t);h(t)),
\end{aligned}
     \end{equation}
     $R^{(n)}(\phi(t);h(t))$ being as in \eqref{eq:FrechetTaylorn} (with $u$ replaced by $\phi$). In this way equation \eqref{TaylorF} can be rearranged as
     \begin{equation} \label{EpsilonTaylorF}
      \begin{aligned}
       & F(\phi(t)+h(t))  \\
       &\quad =  F(\phi(t)) +
      \sum_{j=2}^n  \ep^{j}\left( \sum_{i_1+\dots+i_j=j}^n
        \frac{1}{j!} \nabla^{(j)} F(\phi(t)) [u_{i_1}(t), \dots , u_{i_j}(t)] \right)\\
      & \quad + R_1^{(n)}(\phi(t);h(t),\ep).
      \end{aligned}
    \end{equation}
    \begin{lemma} \label{lm:R1}
     Let $R_1^{(n)}$ be as in formula \eqref{R1}. Then for all $p \in [2, \infty)$ and $T >0$ there exists a constant $C>0$, depending on $| \phi |_{K},\ldots,| u_n |_{H},\nabla^{(1)} F,\ldots, \nabla^{(n)} F,p,n$, such that:
     $$
      \EE\left[\sup_{t \in [0,T]} | R_1^{(n)} (\phi(t);h(t),\ep) |_H^p  \right] \leq C \ep^{p(n+1)}
     $$
     for all $0< \ep \leq 1$.
    \end{lemma}
    \begin{proof}
     First of all we notice that
    $$
       \sum_{j=2}^n {\bf o}_j(\ep^{nj}) = {\bf O}(\ep^{2n}),
     $$
   meaning that
   \begin{align}\label{eq:O}
       \left|\sum_{j=2}^n{\bf o}(\ep^{nj})\right|\leq C_n \ep^{2n}, \qquad \ep \to 0,
\end{align}
    for some constant $C_n>0$.
     Now since:
     \begin{align*}
       &R_1^{(n)} (\phi(t);h(t),\ep) = \sum_{j=2}^n\sum_{i_1+\ldots+i_j  = n+1}^{nj} \ep^{i_1+\dots+i_j}
       \frac{1}{j!} \nabla^{(j)} F(\phi(t)) [u_{i_1}(t), \dots , u_{i_j}(t)] \\
      &  \qquad  \qquad \qquad \qquad + \sum_{j=2}^n {\bf o}_j(\ep^{nj})
       + R^{(n)}(\phi(t);h(t)),
     \end{align*}
      using the estimate given in condition (3.b) in Hypothesis \ref{hp:A+F} and
    \eqref{eq:O}, for $\ep \in (0,1]$ we have
     \begin{equation}
     \begin{aligned}\label{eq:StimaR_1}
     & | R_1^{(n)} (\phi(t);h(t),\ep) |_H^p  \\
    &   \qquad \leq C_{n,p}^{1} \ep^{(n+1)p} \left[ \left( \max_{j=1,\ldots,n} \| \nabla^{(j)} F (\phi(t))\|_{L^j(K)} \right)^p
      \left( \sum_{i=1}^n |u_i(t)|_H^p \right) \right]
 \\
 &\qquad \qquad \qquad +( {\bf O}(\ep^{2n}))^p+ C_{n,p}^2\left| R^{(n)}\left(\phi(t);h(t) \right) \right|_H^p\\
      & \qquad  \leq C_{n,p}^{(1)} \ep^{(n+1)p}  \max_{j=1,\dots,n} \left[ \gamma_j^p(1+|\phi(t)|_K^{m-j})^p \right] \left( \sum_{i=1}^n |u_i(t)|_H^p \right)\\
     & \qquad \qquad \qquad+C_n \ep^{2np}+ C_{n,p}^{(2)}|R^{(n)}(\phi(t);h(t))|_H^p\\
&\qquad  \leq\tilde{C}_n  \ep^{(n+1)p} + C_{n,p}^{(2)}|R^{(n)}(\phi(t);h(t))|_H^p,
      \end{aligned}
       \end{equation}
       where $C^1_{n,p},C_{n,p}^{(1)},C_{n,p}^{(2)}$ are constants depending only on $n,p$ and the constant $C_n$ in \eqref{eq:O} while $\tilde{C}_n$ is a suitable positive constant depending on $p,n, \max_{j=1,\dots,n} \left[ \gamma_j^p(1+|\phi(t)|_K^{m-j})^p \right]$ ($\gamma_i$ being the constants appearing in Hypothesis \ref{hp:A+F}, condition (3))  and $\left| u_i(t) \right|_H^p$, $i=1,\ldots,n$.
        We notice that the above inequality follows by recalling that the deterministic
function $\phi(t)$ is bounded (in the $H$-norm) (see Proposition \ref{prop:MildDeterministica}).

 Now by the bound on $R^{(n)}$ in the equation \eqref{eq:FrechetTaylorn} we have that
$$
        |R^{(n)}(\phi(t);h(t))|_H^p \leq \hat{C}_n |h(t)|_H^{(n+1)p}
       $$
       with $\hat{C}_n$ depending on $\phi(t)$ and $n$ but independent of $h(t)$.
       Since $h(t) = \sum_{k=1}^n \ep^k u_k(t)+r^{(n)}(t;\ep)$ with $|r^{(n)}(t;\ep)|\leq C_n \ep^{n+1}$ for some $\tilde{C}_n$, then:
      \begin{equation}\label{eq:StimaR^n}
        |R^{(n)}(\phi(t);h(t))|_H^p \leq \ep^{(n+1)p} \hat{C}_{n,p}(|u_1(t)|_H,\ldots,|u_n(t)|_H)
       \end{equation}
       with $\hat{C}_{n,p}= \hat{C}_{n,p} (\left|u_1(t)\right|_H, \ldots, \left|u_n(t)\right|_H)$ independent of $\ep$.

       Hence by \eqref{eq:StimaR_1} and \eqref{eq:StimaR^n} we have that
       $$
        \EE\left[\sup_{t \in [0,T]}  | R_1^{(n)} (\phi(t);h(t),\ep) |_H^p \right]
        \leq
        C^\prime_n  \ep^{n+1} ,
       $$
       where $C^\prime_n:=C^\prime_n(p, \nabla^{(1)}F,\ldots,\nabla^{(n)}F,|\phi|_H,\ldots,|u_n|_H )$ is independent of $\ep$. This gives the lemma, with $C = C_n^\prime$.
    \end{proof}

%    Since we would like to concentrate our attention on the parameter $\ep$ we can rearrange $\eqref{TaylorF}$ as follows:
%     \begin{equation} \label{EpsilonTaylorF}
%      \begin{aligned}
%       F(\phi(t)+h(t)) &= F(\phi(t)) +  \sum_{k=2}^n \ep^k
%       \sum_{j=1}^k
%       \sum_{\stackrel{i_1,\ldots,i_j \in \mathbb{N}}{\sum_{l=1}^j i_l = k}}
%       \frac{1}{j!} \nabla^{(j)} F(\phi(t)) [u_{i_1}(t), \dots , u_{i_j}(t)] + \mathcal{R}^{(n)} (t,\ep^{n+1})
%      \end{aligned}
%    \end{equation}
%    where the term $\mathcal{R}^{(n)}(t,\ep^{n+1})$ is given in terms of $R^{(n)}(\phi(t);h(t))$ and powers of $\varepsilon$  strictly greater than $n$.

  %  $R_n^{(1)},\ldots,R_n^{(n)}$.

   % \textcolor{blue}{Soluzione mild problema deterministico e stocastico}

    As we said before, we want to expand the solution of the equation \eqref{eq:eps} around $\phi(t)$, that is we
    want to write $u(t)$ as:
    \begin{equation} \label{eq:Espansioneu(t)}
        u(t)=\phi(t)+\ep u_1(t)+\dots+\ep^n u_n(t)+R_n(t,\ep),
    \end{equation}
    (with the term $R_n(t,\ep) = {\bf O }(\ep^{n+1}) )$, for any $t\geq 0$),
    where the processes $(u_i(t))_{t\geq 0}, i=1,\dots,n$ can be found by using the Taylor expansion of $F$ around $\phi(t)$ and  \textit{matching terms}
    in the equation \eqref{eq:eps} for $u$.
    Given predictable $H$-valued stochastic processes $w(t), v_1(t),\ldots , v_n(t)$ let us use the notation:
     \begin{equation}\label{eq:phik}
 \Phi_k( w(t))\left[v_1(t), \ldots, v_k(t)\right]:= \sum_{j=2}^k \sum_{i_1+\dots+i_j =k} \nabla^{(j)}F(w(t))[v_{i_1}(t),\dots,v_{i_j}(t)]\:,
\end{equation}
with $i_1,\ldots,i_j$, running from $0$ to $k$ and the given restriction $i_1+ \cdots + i_n =k$.
    With the above notation the processes
    $u_1(t),\dots,u_n(t)$ occurring in \eqref{eq:Espansioneu(t)} satisfy the following equations:
    $$
   \begin{cases}
        du_1(t)= [Au_1(t)+\nabla F(\phi(t))[u_1(t)]]dt + \sqrt{Q}dL(t), \\
      u_1(0)=0,
    \end{cases}
     $$
  and
  \begin{equation}\label{eq:StochasticSystem}
    \begin{cases}
        du_k(t)= [Au_k(t)+\nabla F(\phi(t))[u_k(t)]]dt
               + \Phi_k(t)dt, \\
       u_k(0)=0,\\
    \end{cases}
  \end{equation}
    with
    \begin{equation}\label{eq:Phi_k}
     \Phi_k(t) :=\Phi_k( \phi(t))\left[u_1(t), \ldots,u_{k-1}(t) \right] \:,\: k \in \mathbb{N}, n \geq k \geq 2\:.
    \end{equation}
  %  \textcolor{blue}{Nota su gradiente di $F$ in $\phi(t)$}
    Notice that while $u_1(t)$ is the solution of a linear stochastic differential equation (with time dependent drift operator $A+\nabla F(\phi(t))$), the processes $u_2,\ldots,u_n$ are solutions of non-homogenous differential equations with random coefficients whose meaning is given below.
    \begin{definition}\label{def:SolutionUk}
        Let $2 \leq k \leq n$. Then a predictable $H$-valued stochastic process $u_k = u_k(t) \;, t\geq 0$ is a solution of the problem \eqref{eq:uk} $($i.e. \eqref{eq:StochasticSystem}$)$ if almost surely it satisfies the following integral equation
    \begin{equation*}
             u_k(t) = \int_0^t e^{(t-s)A} \nabla F (\phi(s)) [u_k(s)] ds + \int_0^t \Phi_k(s) ds , \qquad t \geq 0 \; , \; 2 \leq k \leq n,
   \end{equation*}
   with $\phi$ as in Proposition $\ref{prop:MildDeterministica}$ and  $\Phi_k$ as in \eqref{eq:phik}
   and \eqref{eq:Phi_k}.
    \end{definition}

In the following result we estimate the norm of $\Phi_k$ in $H$ by
means of the norms of
 the G\^ateaux derivatives of $F$ and the norms of $v_j(t)$, $j=1,\dots,k-1$,  where $v_j(t)$ are $H$-valued stochastic processes.
 \begin{lemma} \label{lemma:phik}
Let us fix $2\leq k\leq n;$ let $w(t)$ and $v_1(t),\ldots,v_{k-1}(t)$ be respectively a $K$-valued process and $H$-valued stochastic processes.
Then $\Phi_k( w(t))\left[v_1(t),\ldots,v_{k-1}(t)\right]$ as in \eqref{eq:phik} satisfies the following inequality
 \begin{equation*}
 \left|\Phi_k( w(t))\left[v_1(t),\ldots,v_{k-1}(t)\right]\right|_H \leq C |w(t)|_K k^2 (k+|v_1(t)|_H^{k-1}+\dots+|v_{k-1}(t)|_H^{k-1}),
 \end{equation*}
 where $C$ is some positive constants depending on $k$ and the constant $\gamma_j$,  $j=2,\ldots,k$
introduced in Hypothesis $\ref{hp:A+F}$.
 \end{lemma}

\begin{proof}
%Let $v_0(t),v_1(t),\ldots,v_{k-1}(t)$ be  $H$-valued stochastic processes, then:
We have
 \begin{equation}
 \begin{aligned}
\left| \Phi_k( w(t))\left[v_1(t),\ldots,v_{k-1}(t)\right] \right|_H &= \left| \sum_{j=2}^k \sum_{i_1+\dots+i_j =k} \frac{\nabla^{(j)}F(w(t)) [v_{i_1}(t),\dots,v_{i_j}(t)]}{j!}
 \right|_H \\
 &\leq \sum_{j=2}^k \sum_{i_1+\dots+i_j =k} \left| \frac{\nabla^{(j)}F(w(t))[v_{i_1}(t),\dots,v_{i_j}(t)]}{j!}  \right|_H
 \end{aligned}
 \end{equation}
 and using the assumption (3) in Hypothesis \ref{hp:A+F}, we get
  \begin{equation}
 \begin{aligned}
 | \Phi_k (t)|_H & \leq
  \sum_{j=2}^k \sum_{i_1+\dots+i_j =k}
  \frac{1}{j!}\| \nabla F^{(j)}(w(t))\|_{L^j(H) }
  \prod_{l=1}^j |v_{i_l}(t)|_H \\
  & \leq
   \sum_{j=2}^k  \frac{1}{j!} \gamma_j(1+ |w(t)|_K)^{m-j}
   \sum_{i_1+\dots+i_j =k}
   \sum_{l=1}^j |v_{i_l}(t)|_H^j \\
   & \leq   \sum_{j=2}^k  \frac{1}{j!} \gamma_j(1+ |w(t)|_K)^{m-j}
   \sum_{i_1+\dots+i_j =k}
    \left( j + \sum_{l=1}^{k-1} |v_l(t)|_H^{k-1} \right)\\
    & \leq
      \sum_{j=2}^k  \frac{1}{j!} \gamma_j(1+ |w(t)|_K)^{m-j}
    k^2
    \left(k + \sum_{l=1}^{k-1} |v_l(t)|_H^{k-1} \right) \\
    & \leq
    C (1+|w(t)|_K^{m-2})  k^2 \left(k + \sum_{l=1}^{k-1} |v_l(t)|_H^{k-1} \right),
 \end{aligned}
 \end{equation}
   for some positive constant $C$, from which the assertion in Lemma \ref{lemma:phik} follows.
\end{proof}
\begin{remark}\rm
Notice that by Lemma \ref{lemma:phik}, if $v_1,\ldots,v_{k-1}$
are $p$-mean ($p \in [2, \infty)$), integrable continuous stochastic processes
then the same holds for $\Phi_k$.
\end{remark}

\section{Main results} \label{sec:3}
\begin{proposition}\label{prop:u1}
    Under Hypothesis $\ref{hp:A+F}$ the following stochastic differential equation$:$
          \begin{equation}\label{eq:u1mr}
         \begin{cases}
         du_1(t)= [Au_1(t)+\nabla F(\phi(t)) [u_1(t)]]dt +  \sqrt{Q}dL(t) , \quad t \in [0,+\infty) \\
         u_1(0) = 0,
    \end{cases}
        \end{equation}
        has, with $\phi$ as in Proposition $\ref{prop:MildDeterministica}$, a unique mild solution satisfying, for any $p\geq2$, the following estimate$:$
        \begin{equation}\label{stimau1}
        \EE\left[ \sup_{t \in [0,T]} | u_1(t) |_H^p  \right]< +\infty, \qquad \textit{for any } T>0.
        \end{equation}
\end{proposition}
    \begin{proof}
     First we show the uniqueness. Let us suppose that $w_1(t)$ and $w_2(t)$
     are two solutions of \eqref{eq:u1mr}. Then by It\^o's formula we have:
     \begin{multline*}
        d |w_1(t)-w_2(t)|_H^2= \left\langle A
         (w_1(t)-w_2(t)),w_1(t)-w_2(t)\right\rangle d t
        \\ + \left\langle \nabla F(\phi(t))[w_1(t)-w_2(t)],
          w_1(t)-w_2(t)\right\rangle d t,
     \end{multline*}
     so that, by the dissipativity condition on $A$ and the estimate
     on $\nabla F$ in Hypothesis \ref{hp:A+F}, (3), we have
     \begin{align*}
        d |w_1(t)-w_2(t)|_H^2 \leq - \omega |w_1(t)-w_2(t)|_H^2 + \gamma_1
        (1+|\phi|_K^{m-1}) |w_1(t)-w_2(t)|_H^2.
     \end{align*}
     Now uniqueness follows by applying Gronwall's lemma.

     As far as the existence is concerned, we proceed by a fixed point argument.
     % For any $p\in [2,\infty)$ and $T>0$, we shall denote by $\Xi_{p,T}$ the space of all (equivalence classes) of predictable $H$-valued processes $v(t)$, $t\geq 0$, such that
    %\begin{equation*}
    %    \|v(t)\|_{p,T}=\left(\EE \sup_{t \in [0,T]} |v(t)|_H^p\right)^{1/p} < \infty.
    %    \end{equation*}
     We introduce the mapping $\Gamma$
      from $\mathcal{L}^p(\Omega;C([0,T];H))$ into itself defined by
      \begin{align*}
        \Gamma(w(t)):= \int_0^t e^{(t-s)A}\nabla F(\phi(s)))[w(s)]d s
        + L_A(t).
        \end{align*}
        We are going to prove that there exists $\tilde{T}>0$ such that $\Gamma$ is a contraction on $\mathcal{L}^p(\Omega;C([0,\tilde{T}];H))$. In fact, for any $v,w \in
       \mathcal{L}^p(\Omega;C([0,\tilde{T}];H))$ we have, for any $0\leq t \leq \tilde{T}$:
        \begin{align*}
           & \| \Gamma(v(t))-\Gamma(w(t))\|^p =
             \EE\left[ \sup_{t\in [0,\tilde{T}]} \left|\int_0^t
             e^{(t-s)A}\nabla F(\phi(s)))[v(s)-w(s)]d s \right|^p_H \right]
             \\
            & \leq  \EE\left[ \sup_{t\in[0,\tilde{T}]}  \int_0^t   \| e^{(t-s)A}\|_{L(H)}^p \left| \nabla F(\phi(s))  [v(s)-w(s)]  \right|_H^p ds\right]\\
          & \leq \EE \left[\sup_{s\in [0,\tilde{T}]} \left| \nabla F(\phi(s))  [v(s)-w(s)]  \right|_H^p  \right] \int_0^{\tilde{T}} \|  e^{(\tilde{T}-s^\prime)A}\|_{L(H)}^p ds^\prime\\
          %& \leq
          % \EE \sup_{[0,\tilde{T}]} \int_0^t e^{-p\omega (t-s)}
          %\left( \| \nabla \|_{\mathcal{L}(H,\mathcal{L}(H,H))}\right)^p
          %|\phi|_H^p |x(s)-y(s) |_H^p ds \\
          & \leq
          \EE\left[ \sup_{s\in [0,\tilde{T}]} |v(s) -w(s) |_H^p \right]\gamma_1^p\left(1+|\phi(s)|_{K}^{m-1}\right)^p
          \frac{1}{\omega p} \left(1- e^{-\omega p \tilde{T}}\right) \\
          & \leq
          \gamma_1^p(1+|u^0|_K^{m-1})^p
          \| v-w \|^p \frac{1}{\omega p} \left(1- e^{-\omega p \tilde{T}}\right),
        \end{align*}
     where we used
condition $(iii)$ in Hypothesis \ref{hp:A+F} for the third
inequality and Proposition \ref{prop:MildDeterministica} for the
last inequality.
        Then if $\tilde{T}$ is sufficiently small (depending on $\omega, p, \gamma_1, \phi$), we see that $\Gamma$  is a contraction on $\mathcal{L}^p(\Omega;C([0,\tilde{T}];H))$.

        By considering the map $\Gamma$ on intervals $[0,\tilde{T}], [\tilde{T},2\tilde{T}], \ldots, [(N-1)\tilde{T},T]$, $\tilde{T} \equiv T/N$, $N \in \mathbb{N}$,
        we have that $\Gamma$ is a contraction on $\mathcal{L}^p(\Omega;C([0,T];H))$ and hence we have the existence and uniqueness of the solution for the equation \eqref{eq:u1mr} in the space $\mathcal{L}^p(\Omega;C([0,T];H))$ for any $p \in [2,\infty)$.

    Let us now consider the estimate \eqref{stimau1}. We write the It\^o formula for the function
$|\cdot|_H^{2a}$ H, applied to the process X. To this end, we recall
the expressions for the first and second derivatives of the function
$H(x):=|x|^{2a}, a \,\in\,\N$.

We have
\begin{align*}
    &\nabla F(x)= 2a |x|^{2(a-1)}x\\
  & \frac{1}{2}{\rm Tr}(Q\nabla F^2(x) )= a {\rm Tr}(Q) |x|^{2(a-1)} + (a-1)a |x|^{2(a-2)} |\sqrt{Q}x|^2.
\end{align*}

Moreover (see, \cite{BoMZ} and \cite{OZS}), we recall that It\^o
formula implies:
\begin{align*}
   \dd F(u(t))= \nabla F(u(t-)) \dd u(t) + \frac{1}{2}{\rm Tr}(Q\nabla F^2(u(t-)) ) \dd u
       + \dd [u](t)
\end{align*}
Although our computations are only formal, they can be justified
using an  approximation argument.
    By condition (iii) in Hypothesis \ref{hp:A+F} we have for all points in the probability space and $p=2a$ with $a \in \mathbb{N}$:
\begin{multline}\label{eq:stimau1d2}
    \dd |u_1(t)|^{2a} = 2a \langle u_1(t-),\dd u_1(t)\rangle_H |u_1(t)|^{2a-2} +
     a {\rm Tr}(Q) |u_1(t-)|^{2(a-1)} \dd t\\+ (a-1)a |u_1(t-)|^{2(a-2)} |\sqrt{Q}u_1(t)|^2\dd t+ \dd [u_1(t)](t).
\end{multline}
By the dissipativity of $A+F$, the first term in the above inequality is estimated by
  \begin{equation}\label{sbadiglio1}
  \begin{aligned}
& \langle u_1(t-),\dd u_1(t)\rangle_H |u_1(t)|^{2a-2}
  \\= &\ \langle A u_1(t), u_1(t) \rangle |u_1(t)|_H^{2a-2} +
      \langle \nabla F(\phi(t))[u_1(t)], u_1(t) \rangle |u_1(t)|_H^{2a-2}
     +  \langle\sqrt{Q} \dd L(t-), u_1(t) \rangle |u_1(t)|_H^{2a-2} %+
%    2a {\rm Tr\, Q} |u_1(t)|_H^{2a-2}\dd |L|(s) \rangle |u_1(t)|_H^{2a-2}   \\
%     &
\\\leq
   &  - \omega
     |u_1(t)|_H^{2a} + 2a \gamma(1+|u^0|_K^{m-1})|u_1(t)|_H^{2a}+2a  \langle\sqrt{Q} \dd L(t), u_1(t)\rangle |u_1(t)|_H^{2a-1}
%   &+ 2a {\rm Tr\, Q} |u_1(t)|_H^{2a-2}\dd |L|(s) \rangle |u_1(t)|_H^{2a-2}  \\
\\      \leq &
     - \tilde{\omega}
     |u_1(t)|_H^{2a} +  \langle \sqrt{Q} L(t), u_1(t)\rangle |u_1(t)|_H^{2a-1},
%+2a {\rm Tr\, Q} |u_1(t)|_H^{2a-2}\dd |L|(s) \rangle |u_1(t)|_H^{2a-2},
   \end{aligned}
   \end{equation}
where $\tilde{\omega}:=\omega-\gamma(1+|u^0|_H)$.

Moreover, the second and third term in \eqref{eq:stimau1d2} can be estimated in the following way:
\begin{align}\label{eq:d2}
       a {\rm Tr}(Q) |u_1(t)|^{2(a-1)} + (a-1)a |u_1(t-)|^{2(a-2)} |\sqrt{Q}u_1(t)|^2
    \leq  C_{a} (\epsilon{\rm Tr }^{2a}(Q)+ \frac{1}{\epsilon}|u_1(t)|^{2a}
    ),
\end{align}
for any $\epsilon\,>0,$ where we used the elementary inequality $a
b^{2(a-1)} \leq C_a (\epsilon a^{2a}+\frac{1}{\epsilon} b^{2a})$,
with $C_a$ being a suitable positive constant. Therefore
\begin{align*}
    %\EE \sup_{t\leq T}
   |u_1(t)|_H^{2a}&  \leq  -2a \left(\tilde{\omega} - \frac{C_a}{\epsilon}\right)
    \int_0^t  |u_1(s)|_H^{2a}\dd \, s + 2a  \int_0^t \langle \sqrt{Q} \dd L(s), u_1(s)\rangle |u_1(s)|_H^{2a-1}\\
   & +C_a \epsilon {\rm Tr }(Q)^{2a}\,T + \int_0^t  {\rm Tr\, Q} \dd |L|(s)
\end{align*}
  and
\begin{multline*}
    \EE \sup_{t\leq T}
   |u_1(t)|_H^{2a}  \leq -(2a \tilde{\omega}  T-\frac{C_a}{\epsilon})
   \EE \sup_{t\leq T}  |u_1(t)|_H^{2a}\\ + 2a \,  \EE \sup_{t\leq T}  \left|\int_0^t \langle \sqrt{Q} \dd L(s), u_1(s)\rangle |u_1(s)|_H^{2a-1} \dd s\right| +  C_{a}\epsilon T + T\int_H {\rm Tr Q} |x|^{2}\nu(\dd x),
\end{multline*}
where we used the relation
\begin{align}\label{eq:dovesitrova}
   \EE \sup_{t\leq T }[u_1](t) \leq \EE \int_0^T {\rm Tr}(Q) \dd [L](t) = \EE \int_0^T  {\rm Tr}(Q)
       \dd \langle L\rangle (t) = T \int_H {\rm Tr}(Q) |x|^2\nu(\dd x).
\end{align}
By the Burkholder-Davis-Gundy inequality, see.e.g,  (\cite{PeZa}, p.
37, \cite{HABU, KALLO}) applied to
\begin{align*}
     M(t):= \int_0^t \langle \sqrt{Q} \dd L(s), u_1(s)\rangle |u_1(s)|_H^{2a-1},
\end{align*}
there exists a constant $c_1$ such that
\begin{align*}
     \EE \sup_{t\leq T}  \left|\int_0^t \langle \sqrt{Q} \dd L(s), u_1(s-)\rangle |u_1(s)|_H^{2a-1} \right| \leq
  & \,c_1  \, \EE \left(\left[ \int_0^{\cdot} \langle \sqrt{Q} \dd L(s), u_1(s-)\rangle |u_1(s)|_H^{2a-1} \right](T)\right)^{1/2} \\
    \leq & \,c_1 \, \EE \left( \sup_{t\leq T} |u_1(t)|^{2a} \int_0^T {\rm Tr}(Q) \dd [L](s)\right)^{1/2}\\
     \leq &\, \ c_1 \, \epsilon \EE \sup_{t\leq T}
   |u_1(t)|_H^{2a} + \frac{c_1 T}{4\epsilon} \int_H {\rm Tr Q} |x|^{2}\nu(\dd x),
\end{align*}
where we used the elementary inequality $ab\leq \epsilon a^2 +
(1/4\epsilon)b^2, \epsilon\,> 0.$ Collecting the above estimates we
obtain
\begin{align*}
     \EE \sup_{t\leq T}
   |u_1(t)|_H^{2a}  &\leq -(2a \tilde{\omega} -\frac{c_a}{\epsilon})T
   \EE \sup_{t\leq T}  |u_1(t)|_H^{2a} \\
    &+ 2c_1 \epsilon \EE \sup_{t\leq T}
   |u_1(t)|_H^{2a} + \left( \frac{c_1}{2\epsilon}+1\right)T  \int_H {\rm Tr Q} |x|^{2}\mu(\dd x) + C_a \epsilon T,
\end{align*}
%   By Hypothesis \ref{prop:StochasticConvolution} we have that:
%   $$
%    \EE\left[  \sup_{t \in [0,T]} | L_A(t) |_H^{2a}\right]  \leq C_a^\prime, \qquad T >0
%   $$
%   (first with $K$ replacing $H$, but then with $H$, due to the assumption on $H,K$) $C_{a}^\prime$ is some positive constant.
%    Integrating on $[0,T]$ both sides in \eqref{sbadiglio1}, taking the expectation of both members in the  inequality and applying Gronwall's lemma to \eqref{sbadiglio1} we obtain:
  Hence
 $$
     \EE \left[ \sup_{t \in [0,T]} |u_1(t)|_H^{2a} \right] \leq C_{a,T}^\prime  e^{-2\, a\,(\tilde{\omega} -c_a/\epsilon)T} < C_{a,T} \:,
   $$
   where $C_{a,T}$ is a positive constant  and \eqref{stimau1} follows.
     %The proof is a slight modification of that given in \cite[Theorem 7.4]{DPZRosso}, in particular we gain the case $p=2$ by using the assumed $m-$dissipativity of $F$.  The proof of estimate \eqref{stimau1} is analogous to the one given in proposition \ref{prop:MildStoc}.
    \end{proof}

    \begin{theorem}\label{thm:uk}
     Let us fix $2\leq k\leq n$, assume that  Hypothesis \ref{hp:A+F} holds, and let $u_1$ be the solution of the problem \eqref{eq:u1}.
  Suppose moreover that   $u_j$ is the unique mild solution of the following Abstract Cauchy Problem (ACP):
      \begin{equation}\label{eq:prb-j}
    \begin{cases}
            du_j(t)= [Au_j(t)+\nabla F(\phi(t))[u_j(t)]]dt,
               + \Phi_j(t)dt \notag \\
               u_j(0) = 0
             \tag{$\rm{ACP}_j$}
\end{cases}
        \end{equation}
       for $j=2,\ldots,k-1$ satisfying$:$
    \begin{equation}\label{stimauj}
        \EE \left[ \sup_{t \in [0,T]} | u_j(t) |_H^p \right]< +\infty, \qquad T>0, \ for \ any \ p\in [2,\infty);
    \end{equation}
        then there exists a unique mild solution $u_k(t)$ of the following non-homogeneous linear differential equation with stochastic coefficients
        $($in the sense of Definition $\ref{def:SolutionUk} ):$
      \begin{equation}\label{eq:prb-k}
   \begin{cases}
            du_k(t)= [Au_k(t)+\nabla F(\phi(t))[u_k(t)]]dt
                + \Phi_k(t)dt, \quad t\in[0,+\infty), \notag \\
                u_k(0) = 0
                \tag{$\rm{ACP}_k$}
\end{cases}
        \end{equation}
        and it satisfies the following estimate, for any $T >0 :$
       \begin{equation}\label{stimauk}
         \EE\left[\sup_{t\in[0,T]} | u_k(t) |_H^p\right] < +\infty .
       \end{equation}
    \end{theorem}
    \begin{proof}
      %The proof of existence and uniqueness on the interval $[0,T]$ for $T>0$ is a slight modification of that given in Proposition \ref{}.
      We proceed by a fixed point argument, where
      the contraction is given by
        \begin{equation*}
          \Gamma(y(t)):=\int_0^{t} e^{(t-s)A} \nabla F (\phi(t))  [y(t)]
           ds  + \int_0^t e^{(t-s)A} \Phi_k(s) ds
        \end{equation*}
        on $\mathcal{L}^p(\Omega;C([0,T];H))$.
        In fact, arguing as in Proposition \ref{prop:u1}, we see that for  $\tilde{T} \in [0,T]$ sufficiently small,
        $\Gamma$ is a contraction on $\mathcal{L}^p(\Omega;C([0,\tilde{T}];H))$, $p \in [2,\infty)$, so that the existence and the uniqueness of the solution
        for \eqref{eq:prb-k} follows.

Let us consider the estimate \eqref{stimauk}.
    By the condition (iv) in Hypothesis \ref{hp:A+F} we have,  for $p=2a$ with $a \in \mathbb{N}$ (and all points in the probability space) :
  \begin{equation}\label{sbadiglio}
  \begin{aligned}
  \frac{d}{dt} |u_k(t)|_H^{2a}
  &= 2a \langle A u_k(t), u_k(t) \rangle |u_k(t)|_H^{2a-2} +
     2a \langle \nabla F(\phi(t))[u_k(t)], u_k(t) \rangle |u_k(t)|_H^{2a-2}\\
     &+ 2a \langle \Phi_k(t), u_k(t) \rangle |u_k(t)|_H^{2a-2} \\
     &\leq
     -2a \omega
     |u_k(t)|_H^{2a} + 2a \gamma(1+|u^0|_K)|u_k(t)|_H^{2a}+2a  |\Phi_k(t)|_H |u_k(t)|_H^{2a-1} \\
     & \leq
     -2a \tilde{\omega}
     |u_k(t)|_H^{2a} + C_a |\Phi_k(t)|_H^{2a},
   \end{aligned}
   \end{equation}
where $\tilde{\omega}:=\omega-\gamma(1+|u^0|_K)$ as in the proof of Proposition \eqref{prop:u1}.
   By the assumption  \eqref{stimauj} made on $u_j(t),j=1,\dots,k-1$ and Lemma \ref{lemma:phik} we have that:
   $$
    \EE\left[  \sup_{t \in [0,T]} | \Phi_k(t) |_H^{2a} \right]\leq C_a^\prime, \qquad T >0,
   $$
   so that taking the expectation of  inequality \eqref{sbadiglio} and applying Gronwall's lemma (similarly as in the proof of Proposition \ref{prop:u1}) we obtain:
   $$
     \EE\left[\sup_{t \in [0,T]} |u_k(t)|_H^{2a}\right] \leq C_a^\prime  e^{-2\,a\,\tilde{\omega} T} < C_a \:,
   $$
   where $C_a$ is a positive constant, and the theorem follows.

    %Using the (stochastic) variation of constants formula we have:
    %\begin{equation}
    %\begin{aligned}
    %    \EE  \left[ \sup_{t \in [0,T]} |u_k(t)|_H^p   \right] & \leq
    %    \EE  \left[ \sup_{t \in [0,T]} \left| \int_0^t e^{(t-s)A} \nabla F(\textbf{u}(s))[u_k(s)] ds \right|_H^p \right] +
    %     \EE  \left[ \sup_{t \in [0,T]} \left| \int_0^t e^{(t-s)A} \Phi_k(s) ds \right|_H^p \right] \\
    %     & \leq
    %     \frac{1}{\omega p} \left(\| \nabla F \|_{\mathcal{L}(H,\mathcal{L}(H,H))} \right)^p
    %     \EE\left[\sup_{t \in [0,T]} \int_0^t  \sup_{s \in [0,t]} |u_k(s)|_H^p ds \right] \\
    %     &+ C |\phi|_H^p k^{4p} \left( \int_0^{T} e^{-2 \omega (T-s)} ds \right) \left( 1+ \sum_{l=1}^{k-1} \EE\left( \sup_{t \in [0,T]} |u_l(t)|_H^{k-1} \right) \right) .
    %\end{aligned}
    %\end{equation}

    \end{proof}

    We are now able to state the main result of this section:
    \begin{theorem}\label{Th:espansione}
         Under Hypothesis $\ref{hp:A+F}$ the mild solution $u(t)$ of \eqref{eq:eps}
         $($in the sense of Definition $\ref{def:stoc-conv})$ can be expanded in powers of $\ep>0$ in the following form
         \begin{equation*}
        u(t)=\phi(t)+\ep u_1(t)+\dots+\ep^n u_n(t)+R_n(t,\ep), \quad n \in \mathbb{N},
    \end{equation*}
   where $u_1$ is the solution of
          \begin{align*}
                 du_1(t)&=[Au_1(t)+\nabla F(\phi(t))[u_1(t)]]dt+\sqrt{Q}dL(t)\\
                 u_1(0)&=0,
            \end{align*}
        while $u_k$, $k=2,\ldots,n$ is the solution of
        \begin{equation}
         \begin{cases}
            du_k(t)= [Au_k(t)+\nabla F(\phi(t))[u_k(t)]dt
               +\Phi_k(t)dt, \notag \\
               u_k(0)=0.
               \end{cases}\tag{$\rm{ACP}_k$}
        \end{equation}
        The remainder $R_n(t,\ep)$ is defined by
 \begin{equation}
    \begin{aligned}
      R_n(t,\ep) & :=  u(t)-\phi(t) - \sum_{k=1}^n \ep^k u_k(t) \\
      &= \int_0^t e^{(t-s)A}
      \left( F(u(s))-F(\phi(s))-\sum_{k=1}^n \ep^k \nabla F (\phi(s))[u_k(s)] - \sum_{k=2}^n \ep^k\Phi_k(s) \right) ds,
      \end{aligned}
    \end{equation}
 and verifies the following  inequality
          \begin{equation*}
\EE \left[\sup_{t\in[0,T]}\left|R_n(t,\ep)\right|_H^p\right] \leq C_p \ep^{n+1},
    \end{equation*}
    with a constant $C_p>0$.
    \end{theorem}
    \begin{proof}
     Let us define $R_n(t,\ep)$, $n \in \mathbb{N}$,  as stated in the theorem.
     Since by construction
     \begin{itemize}
      \item $\phi(t) = e^{t A} u^0 + \int_0^t e^{(t-s)A} F(\phi(s)) ds$ (cf. Definition \ref{def:det});
      \item $u(t) = e^{t A} u^0 + \int_0^t e^{(t-s)A} F(u(s)) ds + \ep L_A(t)$ (cf. Definition \ref{def:stoc-conv});
      \item $u_1(t) = \int_0^t e^{(t-s)A} \nabla F(\phi(s))[u_1(s)] ds + L_A(t)$ (cf. Proposition
     \ref{prop:u1} and Definition \ref{def:stoc-conv});
      \item $u_k(t) = \int_0^t e^{(t-s)A} \nabla F (\phi(s)) [u_k(s)] ds  + \int_0^t e^{(t-s)A} \Phi_k(s) ds$  for $k=2,\ldots,n$, with $\Phi_k(s) := \Phi_k( \phi(s))\left[u_1(s),\ldots, u_{k-1}(s)\right] $ defined in \eqref{eq:Phi_k}
    (cf. Theorem \ref{thm:uk} and Definition \ref{def:stoc-conv});
     \end{itemize}
     we have
     $$
      R_n(t,\ep) = \int_0^t e^{(t-s)A}
      \left( F(u(s))-F(\phi(s))-\sum_{k=1}^n \ep^k \nabla F (\phi(s))[u_k(s)] - \sum_{k=2}^n \ep^k \Phi_k(s) \right) ds \:.
     $$
     Recalling that $R_1^{(n)}(\phi(s);h(s),\ep) = F(u(s))-F(\phi(s))-\sum_{k=1}^n \ep^k \nabla F (\phi(s))[u_k(s)] - \sum_{k=2}^n \ep^k \Phi_k(s)$ we get:
     \begin{equation}
      \begin{aligned}
      \EE\left[ \sup_{t \in [0,T]} \left| R_n(t,\ep) \right|_H^p\right] \leq
       & \EE  \left[\sup_{t \in [0,T]}
       \left| \int_0^t e^{(t-s)A} R_1^{(n)}(\phi(s);h(s),\ep) ds \right|_H^p \right] \\
       \leq
       & \: \EE \left[ \sup_{t\in[0,T]} \int_0^t \| e^{(t-s)A} \|_{L(H)}^p
       |R_1^{(n)}(\phi(s);h(s),\ep)|_H^p ds \right]\\
       \leq
       & \: \EE\left[ \sup_{t\in[0,T]}
        |R_1^{(n)}(\phi(t);h(t),\ep)|_H^p
        \int_0^t e^{-\omega (t-s) p} ds \right]\\
        \leq
        & \: C_{n,p}\ep^{p(n+1)}, %\EE  \sup_{t\in[0,T]} |R_1^{(n)}(\phi(t);h(t),\ep)|_H^p \:,
       \end{aligned}
     \end{equation}
     for some positive constant $C_{n,p}$ (depending on $n,p$, but not on $\ep$), where in the second and third inequality we have used the contraction property of the semigroup generated by $A$. Now  recalling  Lemma \ref{lm:R1} the inequality in Theorem \ref{Th:espansione} follows.
    \end{proof}

   \begin{example} \label{Remark:FHN} \rm

   Our results apply in particular to stochastic PDEs describing the FitzHugh-Nagumo equation with a L\'evy noise perturbation
   (related to those studied with a Gaussian noise, for example, in \cite{Tu1, Tu2, Tu92} and \cite{BoMa}).

   The reference equation is given by (see \cite[equation (1.1)]{BoMa})
   \begin{equation}\label{eq:bm08}
   \begin{cases}
      &\partial_t
     v(t,x)=\partial_x (c(x)\partial_x v(t,x))
       -p(x)v(t,x)-w(t,x)+f(v(t,x))+\ep\dot{L}_1(t,x),\\
       & \partial_t w(t,x)=\gamma v(t,x)-\alpha w(t,x)+\ep \dot{L_2}(t,x), \\
       &\partial_x v(t,0)=\partial_x v(t,1)=0,\\
       %&\partial_x w(t,0)=\partial_x w(t,1)=0\\
       &v(0,x)=v_0(x), \quad w(0,x)=w_0(x),
   \end{cases}
   \end{equation} with the parameter $\ep>0$ in front of the noise, where $u,w$ are real valued random variables, $\alpha, \gamma$ are strictly positive
   real phenomenological constants and $c,p$ are strictly positive smooth functions on $[0,1]$. Moreover, the initial values $v_0, w_0$ are in $C([0,1])$.
    The nonlinear term is of the form $f(v) = -v (v-1) (v- \xi)$, where $\xi \in (0,1)$. Finally $L_1, L_2$ are independent $Q_i$-L\'evy processes with values in
    $L^2(0,1)$, with $Q_i$ positive trace class commuting operators, commuting also with $A_0$, $A_0$ being defined below.
   The above equation can be rewritten in the form of an infinite dimensional
   stochastic evolution equation on the space
 \begin{equation}\label{eq:SettingBoMa}
    H: = L^2 (0,1) \times L^2(0,1)
    \end{equation}
    by introducing the following operators:
    \begin{align*}
       &A_0:=\partial_x c(x)\partial_x, \\
       &D(A_0):=\left\{u \in H^2(0,1); v_x(0)=v_x(1)\right\},\,\, acting\,\, in\,\,\, L^2(0,\,1)\\
    \intertext{and}
     &A =
\begin{pmatrix}
A_0- p  & -I  \\
\gamma I  & \alpha I   \\
\end{pmatrix},
\end{align*}
with domain $D(A):= D(A_0)\times L^2(0,1)$, and
$$
F \binom{v}{w}  =
 \begin{pmatrix} -v (v-1)  (v-\xi)\\ 0 \end{pmatrix} \:,\: \text{with } D(F) := L^6 (0,1) \times L^2(0,1).
$$
Further, we introduce the Banach space $K:=L^{18}(0,1)\times L^2(0,1)$, endowed with the norm
$|\cdot|_K:=|\cdot|_6+|\cdot|_2$ and consider $u^0\in K$.
In this way, the equation \eqref{eq:bm08} can be rewritten as
\begin{align*}
\begin{cases}
   d u(t)= Au(t)+ F(u(t))d t + \sqrt{Q}d L(t)\\
   u(0)=u^0:=(v^0,w^0) \in K
   \end{cases}
   \:,
\end{align*}
with $A$ and $F$ satisfying Hypothesis \ref{hp:A+F} when
 $\xi^2-\xi+1 \leq 3\min_{x\in[0,1]}p(x)$. In fact, the properties of the two operators $A$ and $F$  can be
determined starting from the problems considered in \cite{BoMa} and
\cite{Cerrai99}. In particular from \cite[Section 2.2]{Cerrai99} the
estimates on the nonlinear term $F$ and its derivatives can easily
be deduced. Moreover we claim that the stochastic convolution
\begin{align*}
      L_A(t):= \int_0^t e^{(t-s)A} {\rm d} L(s),
\end{align*}
(where $e^{tA}, t\geq 0$ denotes the semigroup generated by $A$) is
well-defined and admits a continuous version with values into the
space $K$. This fact can be proved by an application of \cite{ PeZa}
and its proof, taking into account that the domain of fractional
powers of $A$ are contained in $K$ (cf. \ Appendix A - in particular
Example A.5.2 - in \cite{DPZRosso}) and moreover we are assuming
${\rm Tr}\,Q<\infty$.

%As a matter of fact in \cite{BoMa} the operator $Q$ in \eqref{eq:eps} is replaced by an operator $ \begin{pmatrix}
% Q_1 & 0  \\
%$0 & Q_2  \\
%\end{pmatrix}   $ with $Q_i$ nuclear operators. However the method of showing existence and uniqueness of mild solutions used in \cite{BoMa} based on \cite[Theorem 5.3.1]{DPZVerde} can be adapted to our situation by exploiting the fact that the linear part $A$ in our case satisfies the stronger assumption (1) in Hyp. \eqref{hp:A+F} which is satisfied in the model described by \eqref{eq:det} due to the fact that the generator $A$ described above
%in (\ref{eq:SettingBoMa}) has discrete spectrum and the Weyl's type estimates for the eigenvalues $\lambda_k$ of $A$, $\lambda_k \sim k^2$ for $k \rightarrow \infty$, hold, A being of the form of a second order strictly elliptic operator with small coefficients on a compact interval, see e.g. \cite{RS4}.

Then by Theorem \ref{Th:espansione} we get an asymptotic expansion in powers of $\ep>0$ of the solution, in terms of solutions of the corresponding
deterministic FitzHugh-Nagumo equation and the solution of a system of (explicit) linear (non homogeneous)  stochastic equations.
The expansion holds for all orders in $\ep>0$.
The remainders are estimated according to Theorem \ref{Th:espansione}.
%We can use these results to carry through a discussion similar to the one made by Tuckwell
These results should allow to obtain rigously results similarly to
those obtained numerically up to second order in $\epsilon$ in
\cite{TuEsp,Tu92} in which the noise was of Gaussian type. Tuckwell,
in particular, has made heuristic expansions up to second order in
$\varepsilon$ for the mean and the variance of the solution process
$u=(u(t))_{t\geq 0}$ (see \cite{TuEsp,Tu92}), proving in particular
that one has enhancement
%This gives in particular a rigorous justification of expansions studied by Tuckwell  up to second order in $\ep$ for the mean and the variance of the solution process \cite{TuEsp}. In particular  it follows from \cite{TuEsp} that one has enhancement
(respectively reduction) of the mean according to whether the expansion is around which stable point of the stationary deterministic equation.\newline

\noindent
%In a future work \cite{ADPM} we shall apply these results to the case of networks of FitzHugh-Nagumo neurons.
%Moreover in the second part of the present work we shall study asymptotic expansions for the case where the dissipativity condition is replaced by other conditions on the non Lipschitz drift term.
   \end{example}

\section*{Acknowledgments}
\noindent This paper was greatly influenced by the research project
NEST at the University of Trento. We thank Stefano Bonaccorsi and
Luciano Tubaro and especially Luca di
Persio for many stimulating discussions.\\
The authors would like to gratefully acknowledge the great
hospitality of various institutions. In particular for the first
author CIRM and the Mathematics Department of the University of
Trento; for him and the third author  King Fahd University of
Petroleum and Minerals at Dhahran; for the second and third author
IAM and HCM at the University of Bonn, Germany.

\medskip

\begin{flushleft}
\footnotesize
{\it S. Albeverio}  \\
 Dept. Appl. Mathematics, University of Bonn,\\
HCM; SFB611, BiBoS, IZKS\\
\medskip
{\it E. Mastrogiacomo}\\
Politecnico di Milano, Dipartimento di Matematica \\
   F. Brioschi, Piazza Leonardo da Vinci 32, 20133 Milano \\
\medskip
{\it B. Smii} \\
Dept. Mathematics, King Fahd University of Petroleum and Minerals, \\
Dhahran 31261, Saudi Arabia\\

\medskip
 E-mail: albeverio@uni-bonn.de \\
\hspace{1,1cm} elisa.mastrogiacomo@polimi.it\\
\hspace{1,1cm}boubaker@kfupm.edu.sa

 \end{flushleft}
\end{document}